\documentclass[review]{elsarticle}

\usepackage{amsmath}
\usepackage{amssymb}
\usepackage{mathrsfs}
\usepackage{graphicx}
\usepackage{amsmath}
\usepackage{latexsym}
\usepackage{amsfonts}
\usepackage{multirow}
\usepackage{cases}
\usepackage{tabularx}
\usepackage{booktabs}
\usepackage{color}
\usepackage{setspace}
\usepackage{algorithm}
\usepackage{algorithmic}

\newtheorem{theorem}{Theorem}
\newtheorem{definition}{Definition}
\newtheorem{proposition}{Proposition}
\newtheorem{lemma}{Lemma}
\newtheorem{example}{Example}
\newenvironment{proof}{{\textbf{Proof.}}\,}{\hfill \qed\\}
\newtheorem{remark}{Remark}
\newtheorem{corollary}{Corollary}

\begin{document}
\begin{frontmatter}

\title{Solvability of Multistage Pseudomonotone Stochastic Variational Inequalities\tnoteref{mytitlenote}}
\tnotetext[mytitlenote]{This work was supported by the National Natural Science
Foundation of China (Grant No. 11771244, 12171271).}

%
\author[1]{Xingbang Cui}
\ead{cxb18@mails.tsinghua.edu.cn}

\address[1]{Department of Mathematical Sciences, Tsinghua University, Beijing 100084, China}

\author[2,3]{Jie Sun}
\ead{jie.sun@curtin.edu.au}
\address[2]{School of Business, National University of Singapore, Singapore 119245}
\address[3]{School of EECMS, Curtin University, Perth WA6845, Australia}

\author[1]{Liping Zhang\corref{mycorrespondingauthor}}
\cortext[mycorrespondingauthor]{Corresponding author}
\ead{lipingzhang@mail.tsinghua.edu.cn}


\begin{abstract}
This paper focuses on the solvability of multistage pseudomonotone stochastic variational inequalities (SVIs). On one hand, some known solvability results of pseudomonotone deterministic variational inequalities cannot be directly extended to pseudomonotone SVIs, so we construct the isomorphism between both and then establish theoretical  results on the existence, convexity, boundedness and compactness of the solution set for pseudomonotone SVIs via such an isomorphism. On the other hand, the progressive hedging algorithm (PHA) is an important and effective method for solving monotone SVIs, but it cannot be directly used to solve nonmonotone SVIs. We propose some sufficient conditions on the elicitability of pseudomonotone SVIs, which opens the door for applying Rockafellar's elicited PHA to solve pseudomonotone SVIs. Numerical results on solving a pseudomonotone two-stage stochastic market optimization problem and randomly generated two-stage pseudomonotone linear complementarity problems are presented to show the efficiency of the elicited PHA for solving pseudomonotone SVIs.
\end{abstract}


\end{frontmatter}

\section{Introduction}


Recently, Rockafellar and Wets \cite{Rockafellar2017} developed the multistage stochastic variational inequality (SVI) model, which may incorperate  recourse decisions and  stagewise disclosure of information. The model  provides a unified framework for describing the optimal conditions of  multistage stochastic optimization problems and stochastic equilibrium problems. Rockafellar and Sun \cite{Rockafellar2019,Rockafellar2020} proposed  progressive hedging algorithms (PHAs) for solving monotone multistage SVI problems and monotone Lagrangian multistage SVI problems. In particular, their algorithms converge linearly in the linear-quadratic setting.  Besides, Chen et al \cite{Chen2017} formulated the two-stage SVI as a two-stage stochastic program with recourse by  the expected residual minimization  procedure, and solved this stochastic program via the sample average approximation (SAA). Pang et al \cite{Pang} systematically studied the theory and the best response method for a  two-stage nonlinear stochastic game model. Some related work to SVIs, multistage stochastic equilibrium problems and PHAs can be found in, e.g., \cite{Chen2005,Chen2009,Fang2007,Gurkan1999,Iusem2017,Iusem2018,Iusem2019,Jiang2008,King1993,Ling2008,Luo2009,Lu2013,Ravat2011,Xu2010,Zhang2009,ZHSY}.

As indicated in \cite{Kannan2019,Rockafellar2019}, the aforementioned literatures on PHAs are mostly limited to the monotone mappings since the PHA is essentially a variety of the proximal point algorithm (PPA) for monotone operators \cite{r1976}.  In practice, however, nonmonotone SVIs such as pseudomonotone SVIs arise frequently \cite{Kannan2019} in various applications such like the competitive exchange economy problem, the stochastic fractional problem and the stochastic product pricing problem. The investigations about solution methods for single-stage pseudomonotone SVIs can be found in \cite{Iusem2017,Kannan2019,Yousefian2018}. A natural question is what about the solvability of the multistage pseudomonotone SVI.

In this paper, we investigate  the solvability of the multistage pseudomonotone SVI including solution basic theory and solution method. On one hand, some remarkable theoretical results on pseudomonotone deterministic VIs can be found in \cite{Facchinei2003}. However, these known results cannot be directly extended to the multistage pseudomonotone SVI. So, we design an isomorphism between the pseudomonotone SVI and  the pseudomonotone  deterministic VI. Via such an isomorphism, we obtain some properties of its solution set including the nonemptiness, compactness, boundedness and convexity. On the other hand, the PHA is a very effective method for solving monotone SVIs, but it cannot be directly used to solve the multistage pseudomonotone SVI. Fortunately, Rockafellar recently introduced the notion of elicitable monotonicity (see definition in Section 2) \cite{Rockafellar2019} in an attempt to extend the  PPA and its varieties from monotone mappings to certain nonmonotone mappings. Motivated by his work, we devote to investigate the  elicitability of the multistage pseudomonotone SVIs for the purpose of applying the PHA to solve the multistage pseudomonotone SVI, which opens the door for applying Rockafellar's elicited PHA \cite{Rockafellar2019} to solve pseudomonotone SVIs.
 As long as the monotonicity of the multistage pseudomonotone  SVI can be elicited, the elicited PHA in \cite{SZ2021,Zhang2019}, which is a specialization of Rockafellar's progressive decoupling algorithm \cite{Rockafellar2019}, can be applied to solve  the multistage pseudomonotone  SVI. Our numerical results indicate that such PHA works pretty well for linear pseudomonotone complementarity problems of ordinary size (i.e., several hundreds of variables and scenarios).

The main contribution of this paper is summarized as follows.
\begin{itemize}
\item[1.] The isomorphism between the pseudomonotone SVI and  the pseudomonotone deterministic VI is presented, which stands as a stepping stone in studying the solvability of pseudomonotone SVIs.
\item[2.] Some properties of the solution set of pseudomonotone SVIs, such as the existence, compactness, boundedness and convexity of solutions, are established.
\item[3.] Some criteria to identify  the elicitability of the pseudomonotone SVIs are provided, which ensure Rockafellar's elicited  PHA can be applied to solve the pseudomonotone SVIs.
\end{itemize}

The rest of this paper is organized as follows. We present the formulation  and some basic notions of the multistage SVI in Section $\ref{pre}$. The isomorphism between the pseudomonotone SVI and  the pseudomonotone deterministic VI is introduced in Section $\ref{iso}$. We derive some results of the solution sets for the pseudomonotone multistage SVI in Section $\ref{exist}$. Some criteria for the elicitability of a pseudomonotone SVI are deduced in  Section $\ref{pha}$. We demonstrate the effectiveness of the elicited PHA by various numerical experiments in Section $\ref{nume}$. The paper is concluded in Section $\ref{con}$.

{\bf Notations:} For any positive integer $m$, $\mathbb{R}^m$ denotes the $m$-dimensional Euclidean space and $\mathcal{L}_m$ denotes the the $m$-dimensional Hilbert space, with $\mathbb{R}^m_+=\{x\in\mathbb{R}^m:~x\ge0\}$. Let $\mathbb{R}^{m\times m}$ denote the set of real matrices of $m$ rows and $m$ columns. For $Q\in\mathbb{R}^{m\times m}$, tr$(Q)$ denotes the trace of $Q$. Given sets $S,U$, int$S$ denotes the interior of $S$, bd$S$ denotes the boundary of $S$, ri$(S)$ denotes the relative interior of $S$, and $S\backslash U$ is the set $\{x:x\in S,\ x\notin U\}$. We use $S_\infty=\{d: x+\tau d\in S,~\forall x\in S,\, \tau\geq0\}$ to designate the recession cone of $S$ and use $S^*=\{d:\left<v,d\right>\geq0,~\forall v\in S\}$ to designate the dual cone of $S$. We also use int$(S_{\infty})^*$ to stand  for int$((S_{\infty})^*)$.

\section{Preliminaries}\label{pre}
\subsection{Formulation of the multistage SVI}
Consider an N-stage sequence
\begin{equation*}
x_1,\xi_1,x_2,\xi_2,\ldots,x_N,\xi_N,
\end{equation*}
where $x_k\in\mathbb{R}^{n_k}$ is the decision vector at the $k$-th stage and $\xi_k\in\Xi_k$ is a random vector with $\Xi_k$ being its support and  becoming known only after $x_k$ is determined. Let $\xi=(\xi_1,\xi_2,\ldots,\xi_N)$ be the random vector defined on the finite sample space $\Xi=\Xi_1\times\Xi_2\times\cdots\times\Xi_N$, where each realization of $\xi$ has a probability $p(\xi)>0$, and these probabilities add up to $1$.

Throughout this paper, define $n=\sum\limits_{k=1}^{N}n_k$ and let $\mathcal{L}_n$ denote the Hilbert space consisting of the mapping from $\Xi$ to $\mathbb{R}^{n}$
\begin{equation*}
x(\cdot):\xi\mapsto x(\xi)=(x_1(\xi)^T,x_2(\xi)^T,\ldots,x_N(\xi)^T)^T,
\end{equation*}
where $x_k(\xi)^T$ denotes the transpose of $x_k(\xi)$ for $k=1,2,\ldots,N$. The inner product of $\mathcal{L}_n$ is defined as
\begin{equation}\label{eqinner}
\langle x(\cdot),w(\cdot)\rangle=\sum\limits_{\xi\in\Xi}p(\xi)\sum\limits_{k=1}^{N}\left<x_k(\xi),w_k(\xi)\right>,\quad x(\cdot),w(\cdot)\in\mathcal{L}_n,
\end{equation}
where $\left<x_k(\xi),w_k(\xi)\right>$ is the Euclidean inner product of $x_k(\xi)$ and $w_k(\xi)$. Further, by restricting mapping $x(\cdot)\in\mathcal{L}_n$ to the following closed subspace
\begin{equation*}
\mathcal{N}_n=\{x(\cdot):x_k(\xi)\ \text{does not depend on}\ \xi_k,\ldots,\xi_N\},
\end{equation*}
we introduce the nonanticipativity constraint on $\mathcal{L}_n$. That is, $x(\cdot)\in\mathcal{N}_n$ means that $x_k(\xi)$ will be influenced by $\xi_1,\ldots,\xi_{k-1}$, but not $\xi_k,\ldots,\xi_N$. The orthogonal complement of $\mathcal{N}_n$ is denoted by
\begin{equation} \label{defmn}
\mathcal{M}_n=\{w(\cdot)\in\mathcal{L}_n:E_{\xi|\xi_1,\ldots,\xi_{k-1}}w_k(\xi)=0,\ k=1,\ldots,N\},
\end{equation}
where $E_{\xi|\xi_1,\ldots,\xi_{k-1}}w_k(\xi)$ is the conditional expectation given   $(\xi_1,\ldots,\xi_{k-1})$. {Moreover, $w(\cdot)$ corresponds to the nonanticipativity multiplier in \cite{Rockafellar2017}, which is understood in stochastic programming as furnishing the shadow price of information. As we will see later, $w(\cdot)$ enables decomposition into a separate problem for each secnario $\xi$.} In addition, the closed convex set $\mathcal{C}\subset\mathcal{L}_n$ is denoted by
\begin{equation} \label{defc}
\mathcal{C}=\{x(\cdot):x(\xi)\in C(\xi),\ \forall \xi\in\Xi\},
\end{equation}
where $C(\xi)$ is a nonempty closed convex set in $\mathbb{R}^n$. The mapping $\mathcal{F}:\mathcal{L}_n\rightarrow\mathcal{L}_n$ is defined as
\begin{equation*}
\mathcal{F}(x(\cdot)):\xi\mapsto F(x(\xi),\xi),\ \forall\xi\in\Xi,
\end{equation*}
where $F(\cdot,\xi)$ is a function from $\mathbb{R}^n$ to $\mathbb{R}^n$. 
The $N$-stage SVI in basic form can be expressed as to find $x(\cdot)$ such that
\begin{equation}\label{eq1}
-\mathcal{F}(x(\cdot))\in N_{\mathcal{C}\cap\mathcal{N}_n}(x(\cdot)),
\end{equation}
where $N_{\mathcal{C}\cap\mathcal{N}_n}(x(\cdot))$ is the normal cone of $\mathcal{C}\cap\mathcal{N}_n$ at $x(\cdot)$. In other words, $$N_{\mathcal{C}\cap\mathcal{N}_n}(x(\cdot))=\{v(\cdot):\left<v(\cdot),y(\cdot)-x(\cdot)\right>\le 0,\ \forall y(\cdot)\in\mathcal{C}\cap\mathcal{N}_n\}.$$

For  convenience, we use SVI($\mathcal{C}\cap\mathcal{N}_n$,$\mathcal{F}$) to denote (\ref{eq1}) and use SOL($\mathcal{C}\cap\mathcal{N}_n$,$\mathcal{F}$) to denote the corresponding solution set. The extensive form of SVI (\ref{eq1}) can be formulated as to find $x(\cdot)\in{\cal N}_n$ and  $w(\cdot)\in{\cal M}_n$ such that
\begin{equation}\label{eq2}
-F(x(\xi),\xi)-w(\xi)\in N_{C(\xi)}(x(\xi)),\quad \forall \xi\in\Xi.
\end{equation}
The following theorem given in \cite{Rockafellar2017} exhibits the relationship between (\ref{eq1}) and (\ref{eq2}).
\begin{theorem}\label{thm1}
 If $x(\cdot)$ solves (\ref{eq2}), then $x(\cdot)$ solves (\ref{eq1}). Conversely, assume that $x(\cdot)$ solves (\ref{eq1}), then $x(\cdot)$ is sure to solve (\ref{eq2}) if there exists some $\hat{x}(\cdot)\in\mathcal{N}_n$ such that $\hat{x}(\xi)\in riC(\xi)$ for all $\xi\in\Xi$ (This is called {\bf  the constraint qualification}).
\end{theorem}

\subsection{Pseudomonotonicity}
We recall the concept of pseudomonotone mapping and its properties.
\begin{definition}\label{def1}
$(a)$ A mapping $F:K\subset \mathbb{R}^{m}\to \mathbb{R}^{m}$ is said to be pseudomonotone if for all $x,y\in K$,
\begin{equation*}
\left<y-x,F(x)\right>\geq0\Rightarrow\left<y-x,F(y)\right>\geq0.
\end{equation*}

$(b)$ Let $\mathcal{K}=\{x(\cdot)|x(\xi)\in K(\xi),\ \forall \xi\in\Xi\}$ be a closed convex set in $\mathcal{L}_n$, where $K(\xi)$ is a closed convex set in $\mathbb{R}^{n}$ for all $\xi\in\Xi$. Then a mapping $\mathcal{F}: \mathcal{K}\subset\mathcal{L}_n \to \mathcal{L}_n$ is said to be
 pseudomonotone on $\mathcal{K}$ if for all $x(\cdot),y(\cdot)\in\mathcal{K}$,
\begin{equation*}
\left<y(\cdot)-x(\cdot),\mathcal{F}(x(\cdot))\right>\geq0\Rightarrow\left<y(\cdot)-x(\cdot),\mathcal{F}(y(\cdot))\right>\geq0.
\end{equation*}
\end{definition}

\begin{remark}\label{rem1}
Obviously, a monotone mapping $F$ must be pseudomonotone while the converse is not necessarily true. The counterexample is $F(x)=1/x$ and $K=[1,2]$.
\end{remark}

Some properties of pseudomonotone functions were given in \cite{Cambini2009}.
\begin{lemma}\label{lem4}
 Let $F$ be a continuously differentiable function defined on an open convex set $S\subset \mathbb{R}^m$. Assume that $F$ is pseudomonotone on $S$, then the following statements hold.

$(a)$ Let $DF(x)$ denote the Jacobian of $F$ at $x$.
\begin{equation}\label{eqlem4}
x\in S, u\in \mathbb{R}^{m}, \left<u,F(x)\right>=0\Rightarrow\left<u,DF(x)u\right>\geq0.
\end{equation}
 Further, if $F(x)\neq0$ for all $x\in S$, then the statement (\ref{eqlem4}) is also sufficient.

$(b)$ Given $x\in S$, the Jacobian $DF(x)$ has at most one negative eigenvalue, where an eigenvalue with multiplicity $k(k>1)$ is counted as $k$ eigenvalues.
\end{lemma}

In order to give some criteria for elicitable monotonicity, we recall some basic results about matrices which were given in \cite{Roger2013}.

\begin{lemma}\label{lem5}
 Let $A,B\in \mathbb{R}^{m\times m}$ be symmetric matrices. Then $AB=BA$ if and only if there exists  orthogonal matrix $Q$ such that $Q^{-1}AQ$ and $Q^{-1}BQ$ are diagonal matrices, where  $Q^{-1}$ is the inverse of $Q$.
\end{lemma}

Let $\Lambda_A=Q^{-1}AQ$ and $\Lambda_B=Q^{-1}BQ$. From Lemma \ref{lem5}, if matrices $A,B\in \mathbb{R}^{m\times m}$ satisfies $AB=BA$, we employ $\lambda_i(A)$ (or $\lambda_i(B)$) to describe the entry located at the $i$-th row and $i$-th column of $\Lambda_A$ (or $\Lambda_B$) with respect to $Q$, as is the eigenvalue of $A$ (or $B$) obviously. 

For $A\in \mathbb{R}^{m\times m}$, we denote the $i$-th largest eigenvalue of $A$ by $\lambda_i^{\downarrow}(A)$. The following lemma is about the comparison of eigenvalues.
\begin{lemma}\label{lem6}
Let $A,B\in \mathbb{R}^{m\times m}$ be symmetric matrices. Then
\begin{equation}\label{eq7}
\lambda_j^{\downarrow}(A)\geq\lambda_i^{\downarrow}(B)+\lambda_{j-i+n}^{\downarrow}(A-B)
\end{equation}
for $1\leq j\leq i\leq m$.
\end{lemma}

We then introduce the diagonally dominant and strictly diagonally dominant matrices \cite{Roger2013}.
\begin{definition}\label{std1}
A matrix $A\in \mathbb{R}^{m\times m}$ is diagonally dominant if
\begin{equation*}
|A_{ii}|\geq\sum\limits_{j\neq i}|A_{ij}|\quad \forall i=1,\ldots,m.
\end{equation*}
It is strictly diagonally dominant if
\begin{equation*}
|A_{ii}|>\sum\limits_{j\neq i}|A_{ij}|\quad \forall i=1,\ldots,m.
\end{equation*}
\end{definition}

\begin{lemma}\label{std2}
Let $A\in \mathbb{R}^{m\times m}$ be strictly diagonally dominant. If $A$ is symmetric and $A_{ii}>0$ for $i=1,\ldots,m$, then $A$ is positive definite.
\end{lemma}

\section{{Isomorphism between $\mathcal{L}_n$ and Euclidean space}}\label{iso}
In this section, we build up an isomorphism between $\mathcal{L}_n$ and the Euclidean space, which is the basis of this paper.

Let $\bar{n}=n\times J$, where $J$ denotes the cardinality of space $\Xi$ defined by $\{\xi^1,\ldots,\xi^J\}$. For $1\le i\le J$, let $x_{p_i}=\sqrt{p(\xi^i)}x(\xi^i)$ and group these $x_{p_i}$ in a vector $x$, i.e.,
 $x=(x_{p_1}^T,x_{p_2}^T,\ldots,x_{p_J}^T)^{T}$.
Define a mapping $\phi: \mathcal{L}_{n}\to \mathbb{R}^{\bar{n}}$ as
\begin{equation}\label{eq1ex}
\phi:\quad x(\cdot)\mapsto x.
\end{equation}
 Clearly, $\phi$ is the isomorphism between $\mathcal{L}_n$ and $\mathbb{R}^{\bar{n}}$.

\begin{lemma}\label{lem1}
Let $\phi$ be the mapping defined by (\ref{eq1ex}). Then $\phi$ is a linear isometric isomorphism between $\mathcal{L}_n$ and $\mathbb{R}^{\bar{n}}$.
\end{lemma}

Before giving some operational formulas, we introduce some related notations. The norm of $x(\cdot)$ in $\mathcal{L}_n$ induced by the inner product (\ref{eqinner}) can be defined as $\|x(\cdot)\|=\sqrt{\left<x(\cdot),x(\cdot)\right>}$. For the ease of statement, we use notation $\|\cdot\|$ to denote both the norm in $\mathcal{L}_n$ and the Euclidean norm in $\mathbb{R}^{\bar{n}}$. In addition, define $B(x_0,\epsilon)=\{x|\|x-x_0\|\leq\epsilon\}$.

\begin{proposition}\label{prop1}
Let $A,B\subset \mathcal{L}_n$ and $\phi$ be defined by (\ref{eq1ex}). Then, we have the following conclusions.
\begin{itemize}
\item[$(a)$] $\phi(A\cup B)=\phi(A)\cup\phi(B)$, $\phi(A\cap B)=\phi(A)\cap\phi(B)$;
\item[$(b)$] $\textrm{int}(\phi(A))=\phi(\textrm{int} A)$, $\textrm{bd}(\phi(A))=\phi(\textrm{bd} A)$;
\item[$(c)$] $\phi(A_{\infty})=\phi(A)_{\infty}$, $\phi(A)^*=\phi(A^*)$;
\item[$(d)$] $\phi(aA+bB)=a\phi(A)+b\phi(B)$, $\forall a,b\in R$;
\item[$(e)$] $\phi(A)$ is convex (or compact) if and only if $A$ is convex (or compact);
\item[$(f)$] $\phi(P_{\mathcal{M}_n}(x(\cdot)))=P_{\phi(\mathcal{M}_n)}(\phi(x(\cdot)))$, where $P_{\mathcal{M}_n}$ (or $P_{\phi(\mathcal{M}_n)}$) is the projection operator onto subspace $\mathcal{M}_n$ defined by (\ref{defmn}) (or $\phi(\mathcal{M}_n)$).
\item[$(g)$] $\phi(N_{\mathcal{C}}(x(\cdot)))=N_{\phi(\mathcal{C})}(\phi(x(\cdot)))$, where  $\mathcal{C}$ is defined by (\ref{defc}).
\end{itemize}\end{proposition}

\begin{proof}
We only need to prove $(b)$, $(c)$ and $(f)$, inasmuch as the other conclusions are obvious.

To prove (b), it suffices to prove $\phi(B(x(\cdot),\epsilon))=B(\phi(x(\cdot)),\epsilon)$. Given $\phi(y(\cdot))$ with $\|y(\cdot)-x(\cdot)\|\leq\epsilon$, we have $\|\phi(y(\cdot))-\phi(x(\cdot))\|\leq\epsilon$ since $\|x(\cdot)\|=\|\phi(x(\cdot))\|$ for any $x(\cdot)\in\mathcal{L}_n$. Thus, $\phi(B(x(\cdot),\epsilon))\subset B(\phi(x(\cdot)),\epsilon)$. Similarly, we obtain $B(\phi(x(\cdot)),\epsilon)\subset\phi(B(x(\cdot),\epsilon))$. Hence, (b) holds.

We now prove (c). Given $d(\cdot)\in A_{\infty}$, we have $x(\cdot)+\lambda d(\cdot)\in A$ for any $\lambda\geq0$ and $x(\cdot)\in A$. It follows that $\phi(x(\cdot))+\lambda\phi(d(\cdot))\in\phi(A)$ for any $\lambda\geq0$ and $x(\cdot)\in A$, which indicates $\phi(d(\cdot))\in\phi(A)_{\infty}$. So we get $\phi(A_{\infty})\subset\phi(A)_{\infty}$. The converse inclusion can be obtained similarly.

To prove (f). Let $y(\cdot)=P_{\mathcal{M}_n}(x(\cdot))$. According to the projection theorem in \cite{Philippe2013}, $\phi(y(\cdot))=P_{\phi(\mathcal{M}_n)}(\phi(x(\cdot)))$ if and only if $\left<\phi(y(\cdot))-\phi(x(\cdot)),\phi(z(\cdot))-\phi(y(\cdot))\right>\ge0$ for all $z(\cdot)\in\mathcal{M}_n$. Further, based on the definition of mapping $\phi$ and Lemma $\ref{lem1}$, we have
\begin{equation} \label{prop-1}
\begin{array}{rcl}
&&\left<\phi(y(\cdot))-\phi(x(\cdot)),\phi(z(\cdot))-\phi(y(\cdot))\right>
= \left<\phi(y(\cdot)-x(\cdot)),\phi(z(\cdot)-y(\cdot))\right>\\
&= & \sum\limits_{1\le i\le J}{\left<\sqrt{p(\xi^i)}(y(\xi^i)-x(\xi^i)),\sqrt{p(\xi^i)}(z(\xi^i)-y(\xi^i))\right>}\\
&= & \sum\limits_{1\le i\le J}{p(\xi^i)\left<y(\xi^i)-x(\xi^i),z(\xi^i)-y(\xi^i)\right>}
= \left<y(\cdot)-x(\cdot),z(\cdot)-y(\cdot)\right>.
\end{array}
\end{equation}
Actually, the last expression in (\ref{prop-1}) is nonnegative as long as we see that $y(\cdot)=P_{\mathcal{M}_n}(x(\cdot))$ and employ the projection theorem again. Thus we obtain the conclusion that $\phi(P_{\mathcal{M}_n}(x(\cdot)))=P_{\phi(\mathcal{M}_n)}(\phi(x(\cdot)))$.
\end{proof}

The mapping $\phi$ also provides an alternative description of SVI (\ref{eq1}). Given a mapping $\mathcal{F}(x(\cdot))$ from $\mathcal{L}_{n}$ to $\mathcal{L}_n$, define
\begin{equation}\label{eq2.5}
\hat{F}(x)=(\sqrt{p(\xi^1)}F(\frac{1}{\sqrt{p(\xi^1)}}x_{p_1},\xi^1)^T,\ldots,\sqrt{p(\xi^J)}F(\frac{1}{\sqrt{p(\xi^J)}}x_{p_J},\xi^J)^T)^T,
\end{equation}
where $x=(x_{p_1}^T,x_{p_2}^T,\ldots,x_{p_J}^T)^{T}$. Obviously, $\phi(\mathcal{F}(x(\cdot)))=\hat{F}(\phi(x(\cdot)))$. Then, SVI (\ref{eq1}) can be transformed into
\begin{equation}\label{eq3}
-\hat{F}(x)\in N_{\phi(\mathcal{C}\cap\mathcal{N}_n)}(x).
\end{equation}
We use VI($\phi(\mathcal{C}\cap\mathcal{N}_n)$,$\hat{F}$) to denote the problem $(\ref{eq3})$, and its solution set is denoted as SOL($\phi(\mathcal{C}\cap\mathcal{N}_n)$,$\hat{F}$).

\begin{lemma} \label{lem3}
Let $\phi$ be defined by (\ref{eq1ex}). Then $\phi$ establishes a one-to-one correspondence between SOL($\mathcal{C}\cap\mathcal{N}_n$,$\mathcal{F}$) and SOL($\phi(\mathcal{C}\cap\mathcal{N}_n)$,$\hat{F}$). In other words, if $x(\cdot)$ is a solution of (\ref{eq1}), then $\phi(x(\cdot))$ is a solution of (\ref{eq3}), and vice versa.
\end{lemma}
\begin{proof}
Assume that $x(\cdot)$ is the solution for (\ref{eq1}). Let $x=\phi(x(\cdot))$. It suffices to show that $\left<\hat{F}(x),y-x\right>\geq0$ for all $y=\phi(y(\cdot))\in\phi(\mathcal{C}\cap\mathcal{N}_n)$. Actually,
\begin{align*}
\left<\hat{F}(x),y-x\right>
= & \sum\limits_{1\leq i\leq J}{\left<\sqrt{p(\xi^i)}F(\frac{1}{\sqrt{p(\xi^i)}}x_{p_i},\xi^i),y_{p_i}-x_{p_i}\right>}\\
= & \sum\limits_{1\leq i\leq J}{\left<\sqrt{p(\xi^i)}F(x(\xi^i),\xi^i),\sqrt{p(\xi^i)}y(\xi^i)-\sqrt{p(\xi^i)}x(\xi^i)\right>}\\
= & \left<\mathcal{F}(x(\cdot)),y(\cdot)-x(\cdot)\right>\\
\geq & 0.
\end{align*}
To prove the converse conclusion, we only need to reverse the above deductions.
\end{proof}

\section{Properties of solution sets  of pseudomonotone  SVIs} \label{exist}
In this section, we generalize  the solution theory of pseudomonotone determined VIs given in \cite{Facchinei2003} to multistage pseudomonotone  SVIs. { We first introduce the results from \cite{Facchinei2003} in the following theorems.
\begin{theorem}\label{F1}
Let $K\subset \mathbb{R}^{m}$ be closed convex and  $F:K\rightarrow \mathbb{R}^{m}$ be continuous. Assume that $F(x)$ is pseudomonotone on $K$. Then the following three statements $(a), (b)$ and $(c)$ are equivalent. Moreover, if there exists some $\hat{x}\in K$ such that the set
\begin{equation*}
	L_{\leq}=\{x\in K:\left<F(x),x-\hat{x}\right>\leq0\}
\end{equation*}
is bounded, then SOL($K$,$F$) is nonempty and compact.

\begin{itemize}
	\item[$(a)$] There exists $\hat{x}\in K$ such that the set
	\begin{equation*}
		L_<=\{x\in K:\left<F(x),x-\hat{x}\right><0\}
	\end{equation*}
	is bounded (possibly empty).\\
	\item[$(b)$] There exist a bounded open set $\Omega$ and some $\hat{x}\in K\cap\Omega$ such that
	\begin{equation*}
		\left<F(x),x-\hat{x}\right>\geq0,\quad \forall x\in K\cap bd\Omega.
	\end{equation*}
	\item[$(c)$] VI($K$,$F$) has a solution.
\end{itemize}
\end{theorem}
\begin{theorem}\label{F2}
	Let $K\subset \mathbb{R}^{m}$ be closed convex and  $F:K\rightarrow \mathbb{R}^{m}$ be continuous. If $F$ is pseudomonotone on $K$, then the following statements hold.
	\begin{itemize}
		\item[$(a)$] The solution set SOL($K$,$F$) is convex.
		\item[$(b)$] If there exists $\hat{x}\in K$ such that $F(\hat{x})\in {\rm int}(K_\infty)^*$,  then SOL($K$,$F$) is nonempty, convex and compact.
\end{itemize}\end{theorem}
\begin{theorem}\label{F3}
Let $K\subset \mathbb{R}^{m}$ be closed convex and  $F:K\rightarrow \mathbb{R}^{m}$ be continuous. Assume that $F$ is pseudomonotone on $K$. Then the set SOL($K$,$F$) is nonempty and bounded if and only if
	\begin{equation*}
		K_\infty\cap[-(F(K)^*)]=\{0\}.
	\end{equation*}
\end{theorem}

Theorems \ref{F1}, \ref{F2} and \ref{F3} do not pertain to the pseudomonotne SVI $(\ref{eq1})$ due to the fact that they are based on the Euclidean space. Nonetheless, this extension follows the isomorphim introduced in Section \ref{iso}, which are presented below.}

\begin{theorem} \label{thm2}
Let $\mathcal{F}(x(\cdot)):\mathcal{L}_{n}\rightarrow\mathcal{L}_n$ be continuous. Assume that $\mathcal{F}(x(\cdot))$ is pseudomonotone on $\mathcal{C}\cap\mathcal{N}_n$. Then the following three statements $(a), (b)$ and $(c)$ are equivalent. Moreover, if there exists some $\hat{x}(\cdot)\in\mathcal{C}\cap\mathcal{N}_n$ such that the set
\begin{equation*}
L_{\leq}=\{x(\cdot)\in\mathcal{C}\cap\mathcal{N}_n:\left<\mathcal{F}(x(\cdot)),x(\cdot)-\hat{x}(\cdot)\right>\leq0\}
\end{equation*}
is bounded, then SOL($\mathcal{C}\cap\mathcal{N}_n$,$\mathcal{F}$) is nonempty and compact.

\begin{itemize}
\item[$(a)$] There exists $\hat{x}(\cdot)\in\mathcal{C}\cap\mathcal{N}_n$ such that the set
\begin{equation*}
L_<=\{x(\cdot)\in\mathcal{C}\cap\mathcal{N}_n:\left<\mathcal{F}(x(\cdot)),x(\cdot)-\hat{x}(\cdot)\right><0\}
\end{equation*}
is bounded (possibly empty).\\
\item[$(b)$] There exist a bounded open set $\Omega$ and some $\hat{x}(\cdot)\in\mathcal{C}\cap\mathcal{N}_n\cap\Omega$ such that
\begin{equation*}
\left<\mathcal{F}(x(\cdot)),x(\cdot)-\hat{x}(\cdot)\right>\geq0,\quad \forall x(\cdot)\in\mathcal{C}\cap\mathcal{N}_n\cap bd\Omega.
\end{equation*}
\item[$(c)$] SVI($\mathcal{C}\cap\mathcal{N}_n$,$\mathcal{F}$) has a solution.
\end{itemize}
\end{theorem}

\begin{proof} We first prove the statements (a), (b) and (c) are equivalent.

$``(a)\Rightarrow(b)"$. Let $\Omega$ be a bounded open set containing $L_<\cup\{\hat{x}(\cdot)\}$. It suffices to see that $L_<\cap bd\Omega=\emptyset$.

$``(b)\Rightarrow(c)"$. Denote $x=\phi(x(\cdot))$ with $x(\cdot)\in\mathcal{C}\cap\mathcal{N}_n\cap\Omega$, and $\hat{x}=\phi(\hat{x}(\cdot))$. Following the same procedure in Lemma \ref{lem3}, we have
\begin{equation*}
\left<\hat{F}(\hat{x}),x-\hat{x}\right>\geq0,\quad \forall x\in\phi(\mathcal{C}\cap\mathcal{N}_n)\cap \textrm{bd}(\phi(\Omega)).
\end{equation*}
Then, SOL($\phi(\mathcal{C}\cap\mathcal{N}_n)$,$\hat{F}$) is nonempty by means of Theorem \ref{F1}. In terms of Lemma \ref{lem3}, SVI($\mathcal{C}\cap\mathcal{N}_n$,$\mathcal{F}$) has a solution.

$``(c)\Rightarrow(a)"$. Let $\hat{x}(\cdot)$ be a solution of SVI($\mathcal{C}\cap\mathcal{N}_n$,$\mathcal{F}$). Then, $$\left<\mathcal{F}(\hat{x}(\cdot)),x(\cdot)-\hat{x}(\cdot)\right>\geq0,\quad \forall x(\cdot)\in\mathcal{C}\cap\mathcal{N}_n.$$
 Since $\mathcal{F}$ is pseudomonotone on $\mathcal{C}\cap\mathcal{N}_n$, we have $\left<\mathcal{F}(x(\cdot)),x(\cdot)-\hat{x}(\cdot)\right>\geq0$ for any $x(\cdot)\in\mathcal{C}\cap\mathcal{N}_n$, as indicates the emptiness of $L_<$.

 Since there exists some $\hat{x}(\cdot)\in\mathcal{C}\cap\mathcal{N}_n$ such that the set such that $L_{\le}$ is bounded, (a) holds due to $L_<\subset L_{\leq}$. Thus, (c) holds, which implies that  SOL($\mathcal{C}\cap\mathcal{N}_n$,$\mathcal{F}$) is nonempty. Furthermore, SOL($\mathcal{C}\cap\mathcal{N}_n$,$\mathcal{F}$) is compact because SOL($\mathcal{C}\cap\mathcal{N}_n$,$\mathcal{F}$) is closed and SOL$(\mathcal{C}\cap\mathcal{N}_n,\mathcal{F})\subset L_{\leq}$.
\end{proof}

\begin{theorem}\label{thm3}
Let $\mathcal{F}(x(\cdot)):\mathcal{L}_{n}\rightarrow\mathcal{L}_n$ be continuous. Assume that $\mathcal{F}(x(\cdot))$ is pseudomonotone on $\mathcal{C}\cap\mathcal{N}_n$, the following statements hold.
\begin{itemize}
\item[$(a)$] The solution set SOL($\mathcal{C}\cap\mathcal{N}_n$,$\mathcal{F}$) is convex.
\item[$(b)$] If there exists $\hat{x}(\cdot)\in\mathcal{C}\cap\mathcal{N}_n$ such that $\mathcal{F}(\hat{x}(\cdot))\in {\rm int}((\mathcal{C}\cap\mathcal{N}_n)_\infty)^*$,  then SOL($\mathcal{C}\cap\mathcal{N}_n$,$\mathcal{F}$) is nonempty, convex and compact.
\end{itemize}\end{theorem}

\begin{proof}
$``(a)"$ Obviously, if $\mathcal{F}$ is pseudomonotone on $\mathcal{C}\cap\mathcal{N}_n$, then $\hat{F}$ is also pseudomonotone on $\phi(\mathcal{C}\cap\mathcal{N}_n)$. According to Theorem \ref{F2}, SOL($\phi(\mathcal{C}\cap\mathcal{N}_n)$,$\hat{F}$) is convex. Then SOL($\mathcal{C}\cap\mathcal{N}_n$,$\mathcal{F}$) is convex by Lemma \ref{lem3} and Proposition \ref{prop1}.

$``(b)"$ If $\mathcal{F}(\hat{x}(\cdot))\in {\rm int}((\mathcal{C}\cap\mathcal{N}_n)_\infty)^*$, then $\phi(\mathcal{F}(\hat{x}(\cdot)))\in\phi({\rm int}((\mathcal{C}\cap\mathcal{N}_n)_\infty)^*)$. Let $\hat{x}=\phi(\hat{x}(\cdot))$. Due to the definition of $\hat{F}$ and Proposition \ref{prop1}, we have $\hat{F}(\hat{x})\in {\rm int}((\phi(\mathcal{C}\cap\mathcal{N}_n))_\infty)^*$. Then the fact that SOL($\phi(\mathcal{C}\cap\mathcal{N}_n)$,$\hat{F}$) is nonempty, convex and compact follows from \ref{F2}. It follows from Proposition \ref{prop1} that SOL($\mathcal{C}\cap\mathcal{N}_n$,$\mathcal{F}$) is nonempty, convex and compact.
\end{proof}

\begin{theorem}\label{thm4}
Let $\mathcal{F}(x(\cdot)):\mathcal{L}_{n}\rightarrow\mathcal{L}_n$ be continuous. Assume that $\mathcal{F}(x(\cdot))$ is pseudomonotone on $\mathcal{C}\cap\mathcal{N}_n$. Then, the set SOL($\mathcal{C}\cap\mathcal{N}_n$,$\mathcal{F}$) is nonempty and bounded if and only if
\begin{equation}\label{eqthm4}
(\mathcal{C}\cap\mathcal{N}_n)_\infty\cap[-(\mathcal{F}(\mathcal{C}\cap\mathcal{N}_n)^*)]=\{0\}.
\end{equation}
\end{theorem}
\begin{proof}
By Lemma \ref{lem3} and Proposition \ref{prop1}, SOL($\mathcal{C}\cap\mathcal{N}_n$,$\mathcal{F}$) is nonempty and bounded if and only if SOL($\phi(\mathcal{C}\cap\mathcal{N}_n)$,$\hat{F}$) is nonempty and bounded. On the other hand, if $\mathcal{F}$ is pseudomonotone on $\mathcal{C}\cap\mathcal{N}_n$, then $\hat{F}$ is pseudomonotone on $\phi(\mathcal{C}\cap\mathcal{N}_n)$. Based on Theorem \ref{F3}, SOL($\phi(\mathcal{C}\cap\mathcal{N}_n)$,$\hat{F}$) is nonempty and bounded if and only if ${(\phi(\mathcal{C}\cap\mathcal{N}_n))}_\infty\cap([-\hat{F}(\phi(\mathcal{C}\cap\mathcal{N}_n))^*])=\{0\}$, as is equivalent to (\ref{eqthm4}) due to Proposition \ref{prop1}. Thus we complete the proof.
\end{proof}

\section{Finding solutions to SVI via PHA} \label{pha}
\subsection{{ Description of elicited PHA}}
As is presented in Remark \ref{rem1}, the monotone mapping $\mathcal{F}$ must be pseudomonotone, but the converse is not necessarily true. Thus the original PHA in \cite{Rockafellar2019} can not be applied to the pseudomonotone SVI (\ref{eq1}) directly. Instead,  we employ the elicited PHA proposed in  \cite{Zhang2019} to solve the pseudomonotone SVI (\ref{eq1}), which is motivated by the work of Rockafellar in \cite{Rockafellar2018elicited}.

 We first introduce the concept of { (global) maximal monotonicity in \cite{Rockafellar1998}} and (global) elicited monotonicity in \cite{Rockafellar2018elicited,Zhang2019}.
{
\begin{definition}
 Denote the graph of set-valued mapping $T:H \rightrightarrows H$ by $gphT=\{(x,y):y\in T(x)\}$. Then a monotone mapping $T:H \rightrightarrows H$ is maximal monotone if no enlargement of $gphT$ is possible in $H\times H$ without destroying monotonicity, or in other words, if for every pair $(\hat{x},\hat{y})\in(H\times H)\backslash gphT$ there exists $(\tilde{x},\tilde{y})\in gphT$ with $\left<\hat{y}-\tilde{y},\hat{x}-\tilde{x}\right><0$, where $H$ stands for the Hilbert space.
\end{definition}}

\begin{definition}\label{def2}
 Let $\mathcal{F}:\mathcal{L}_n \to\mathcal{L}_n$. Given closed convex set $\mathcal{C}$, subspace $\mathcal{N}_n$ and its complement $\mathcal{M}_n$, the monotonicity of $\mathcal{F}+N_{\mathcal{C}}$ is said to be elicited at level $s>0$ (with respect to $\mathcal{N}_n$))
if $\mathcal{F}+N_{\mathcal{C}}+sP_{\mathcal{M}_n}$ is maximal monotone globally.
\end{definition}

Despite the fact that $\mathcal{F}+N_{\mathcal{C}}$ may be not maximal monotone, $\mathcal{F}+N_{\mathcal{C}}+sP_{\mathcal{M}_n}$ may be maximal monotone for some $s>0$. So we can apply the PPA to $\mathcal{F}+N_{\mathcal{C}}+sP_{\mathcal{M}_n}$, which is the core idea of the elicited PHA. We refer readers to \cite{Rockafellar2018elicited,Zhang2019} for more details about the elicited PHA.

We intend to use the elicited PHA to solve the pseudomonotone multistage SVIs. { The algorithmic framework and the convergence analysis of the elicited PHA in \cite{Zhang2019} are listed below for completeness.}

\begin{algorithm}\label{alg1}
\caption{Elicited PHA for multistage pseudomonotone SVIs}
\begin{algorithmic}
\STATE{ {\bf Initialization.} Given parameter $r>s\geq0$, $x^0(\cdot)\in{\cal N}_n$ and $w^0(\cdot)\in{\cal M}_n$. Set $k=0$.}
\STATE{ {\bf Iterations.}
  \begin{itemize}
    \item[]{\bf Step 1} For each $\xi\in\Xi$, find $\hat{x}^k(\xi):=$ the unique $x(\xi)$ such that
    \begin{equation}\label{eqalg1}
    -F(x(\xi),\xi)-w^k(\xi)-r[x(\xi)-x^k(\xi)]\in N_{\mathcal{C}(\xi)}(x(\xi)).
    \end{equation}
    \item[] {\bf Step 2} {\bf (primal update).} $x^{k+1}(\cdot)=P_{\mathcal{N}_n}({\hat{x}}^k(\cdot))$.
    \item[] {\bf Step 3} {\bf (dual update).} $w^{k+1}(\cdot)=w^k(\cdot)+(r-s)({\hat{x}}^k(\cdot)-x^{k+1}(\cdot))$.
  \end{itemize}}
\STATE{Set $k:=k+1$; repeat until a stopping criterion is met.}
\end{algorithmic}
\end{algorithm}

\begin{theorem}
 Suppose that $\mathcal{F}+N_{\mathcal{C}}$ in SVI (\ref{eq1}) is globally elicited at level $s$ and the constraint qualification holds. As long as SVI (\ref{eq1}) has a solution, the sequence $\{x^k(\cdot),\omega^k(\cdot)\}$ generated by Algorithm \ref{alg1} will converge to some pair $(\bar{x}(\cdot),\bar{\omega}(\cdot))$ satisfying (\ref{eq2}) and thus furnish $\bar{x}(\cdot)$ as a solution to (\ref{eq1}). In the special case that $\mathcal{F}$ is linear and $\mathcal{C}$ is polyhedral, the convergence will surely be at a linear rate with respect to the norm
\begin{equation*}
{\Vert(x(\cdot),\omega(\cdot))\Vert}^2_{r,s}={\Vert x(\cdot)\Vert}^2+\frac{1}{r(r-s)}{\Vert\omega(\cdot)\Vert}^2.
\end{equation*}
\end{theorem}

\subsection{Some criteria for the elicited monotonicity}
In this subsection, we will provide some criteria for the elicited monotonicity of $\mathcal{F}+N_{\mathcal{C}}$ in SVI (\ref{eq1}). At first, we provide one useful fact about the eigenvalues of the Jacobian of pseudomonotone function.

\begin{lemma}\label{lem7}
Let $F:S\rightarrow \mathbb{R}^{m}$ be a continuously differentiable function defined on an open convex set $S\subset \mathbb{R}^m$. Assume that $F$ is pseudomonotone on $S$ Then the Jacobian $DF(x)$  has at most one negative eigenvalue if it is symmetric, where an eigenvalue with multiplicity $k\ (k>1)$ is counted as $k$ eigenvalues.
\end{lemma}
\begin{proof}
According to \cite[Theorem 1.3.1]{Facchinei2003}, there is a real-valued function $\theta: \mathbb{R}^{m}\rightarrow \mathbb{R}$ such that $\nabla\theta(x)=F(x)$ for $x\in S$. By \cite[Theorems 3.4.1, 5.5.2]{Cambini2009}, the Hessian matrix $\nabla^{2}\theta(x)$ has at most one negative eigenvalue. So $DF(x)\equiv\nabla^{2}\theta(x)$ has at most one negative eigenvalue.
\end{proof}

\begin{remark}\label{lem7-1}
The condition ``$DF(x)$ is symmetric" in Lemma \ref{lem7} is called \emph{the symmetry condition}. Actually, the symmetry condition holds true if and only if there is a real-valued differentiable function $\theta(x)$ such that $\nabla\theta(x)=F(x)$ for all $x\in S$, which is equivalent to \emph{the integrability condition} that $F$ is integrable on $S$. For more details, refer to \cite[Theorem 1.3.1]{Facchinei2003}.

There are two important examples where $DF(x)$ is symmetric. One is the separable function, i.e., $F(x)=(F_{i}(x_{i}):i=1,2,\ldots,m)$, where $F_{i}(x_{i})$, the component function of $F(x)$, only depends on the single variable $x_{i}$. If $F(x)$ is separable and differentiable, then Jacobian $DF(x)$ is diagonal for all $x\in S$. The other one is the linear function $F(x)=Mx+q$ with $M\in \mathbb{R}^{m\times m}$ being symmetric and $q\in \mathbb{R}^m$. In this case, $DF(x)=M$ is symmetric for $x\in S$.
\end{remark}

{The first criterion is based on \cite[Theorem 6]{Rockafellar2018elicited}, which is shown in the theorem below.
\begin{theorem}\label{thmelicit}
Let $H$ be the Hilbert space, and $N_s$ and $M_s$ be the orthogonal subspaces of $H$. Denote $T=F+N_{C}$ for a nonempty closed convex subset $C\subset H$ and a continuously differentiable mapping $F:C\rightarrow H$ with Jacobians $DF(x)$. Assume that there exists $\alpha(x)$ such that when $x\in C$
\begin{equation*}
	\left<y,DF(x)y\right>\geq\alpha(x){\Vert y\Vert}^2,\quad \forall\ y\in N_s.
\end{equation*}
Let
\begin{equation*}
	\beta(x)=\frac{1}{2}\Vert P_{N_s}(DF(x)+(DF(x))^T)P_{M_s}\Vert,\quad
	\gamma(x)=\Vert P_{M_s}DF(x)P_{M_s}\Vert.
\end{equation*}
Define
\begin{equation*}
	\tilde{e}=\sup\limits_{x\in C\cap N_s}\left\{\frac{\beta^2(x)}{\alpha(x)}+\gamma(x)\right\}.
\end{equation*}
If $\tilde{e}<+\infty$, then the monotonicity of $\mathcal{F}+N_{\mathcal{C}}$ is globally elicited at level $s>\tilde{e}$ with respect to $N_s$.
\end{theorem}

Similarly, the isomorphism in Section \ref{iso} enables us to extend Theorem \ref{thmelicit} to pseudomonotone SVI (\ref{eq1}).}

\begin{theorem}\label{thm5}
Consider SVI (\ref{eq1}). Let $\mathcal{F}:\mathcal{C}\rightarrow\mathcal{L}_n$ be continuously differentiable and pseudomonotone on $\mathcal{C}$. Let $\hat{F}$ be defined in (\ref{eq2.5}). Assume that there exists $\alpha(x)$ such that when $x\in\phi(\mathcal{C})$
\begin{equation} \label{eq4}
\left<y,D{\hat{F}}(x)y\right>\geq\alpha(x){\Vert y\Vert}^2,\quad \forall\ y\in\phi(\mathcal{N}_n).
\end{equation}
 Let
\begin{equation*}
\beta(x)=\frac{1}{2}\Vert P_{\phi(\mathcal{N}_n)}(D\hat{F}(x)+(D\hat{F}(x))^T)P_{\phi(\mathcal{M}_n)}\Vert,\quad
\gamma(x)=\Vert P_{\phi(\mathcal{M}_n)}D\hat{F}(x)P_{\phi(\mathcal{M}_n)}\Vert.
\end{equation*}
Define
\begin{equation*}
e_0=\sup\limits_{x\in\phi(\mathcal{C})\cap\phi(\mathcal{N}_n)}\left\{\frac{\beta^2(x)}{\alpha(x)}+\gamma(x)\right\}.
\end{equation*}
If $e_0<+\infty$, then the monotonicity of $\mathcal{F}+N_{\mathcal{C}}$ is globally elicited at level $s>e_0$.
\end{theorem}
\begin{proof}
Let $\mathcal{T}_s=\mathcal{F}+N_{\mathcal{C}}+sP_{\mathcal{M}_n}$ and ${\hat{\mathcal{T}}}_s=\hat{F}+N_{\phi(\mathcal{C})}+sP_{\phi(\mathcal{M}_n)}$.  By Theorem \ref{thmelicit} and the hypotheses, ${\hat{\mathcal{T}}}_s$ is globally maximal monotone for $s>e_0$.

Suppose by  contradiction that  there exists $s'>e_0$ such that $\mathcal{T}_{s'}$ is not maximal monotone. Then there is a pair $(x_u(\cdot),y_u(\cdot))\in\mathcal{L}_n\times\mathcal{L}_n$ with $y_u(\cdot)\notin\mathcal{T}_{s^{'}}(x_u(\cdot))$ such that
\begin{equation}\label{eqzlp}
\left<x_u(\cdot)-x_v(\cdot),y_u(\cdot)-y_v(\cdot)\right>\geq0, \quad \forall (x_v(\cdot),y_v(\cdot))~\text{with $y_v(\cdot)\in\mathcal{T}_{s'}(x_v(\cdot))$}.
\end{equation}
Let $x_u=\phi(x_u(\cdot))$, $x_v=\phi(x_v(\cdot))$, $y_u=\phi(y_u(\cdot))$ and $y_v=\phi(y_v(\cdot))$. Obviously, $y_u\notin{\hat{\mathcal{T}}}_{s'}(x_u)$ and $y_v\in{\hat{\mathcal{T}}}_{s'}(x_v)$. Following the proof of Lemma \ref{lem3}, (\ref{eqzlp}) yields
\begin{equation*}
\left<x_u-x_v,y_u-y_v\right>\geq0,
\end{equation*}
which implies that ${\hat{\mathcal{T}}}_{s'}$ is not globally maximal monotone. This is a contradiction and hence we complete the proof.
\end{proof}

\begin{remark} \label{rem3} In the following two special cases, we can simplify the condition (\ref{eq4}) in Theorem \ref{thm5}.

{\bf Case I:} Denote ${\rm Ker}(\hat{F}(x)):=\{z:\left<z,\hat{F}(x)\right>=0\}.$ Assume that $F$ is pseudomonotone on $S':=\{P_{\phi(\mathcal{N}_n)}z:z\in R^{\bar{n}}\backslash {\rm Ker}(\hat{F}(x))\}$. By Lemma \ref{lem4}, the condition (\ref{eq4}) can be simplified as follows:
\begin{equation}\label{eq5}
\begin{array}{rcl}
&&\left<y,D{\hat{F}}(x)y\right>\geq\alpha(x){\Vert y\Vert}^2,\quad \forall y\in S',\ \mbox{when}\ x\in {\rm int}\phi(\mathcal{C}),\\
&&\left<y,D{\hat{F}}(x)y\right>\geq\alpha(x){\Vert y\Vert}^2,\quad \forall y\in \phi(\mathcal{N}_n),\ \mbox{when}\ x\in {\rm bd}\phi(\mathcal{C}).
\end{array}
\end{equation}

 {\bf Case II:} Let $S':=\{P_{\phi(\mathcal{N}_n)}z:z\in R^{\bar{n}}\backslash {\rm Ker}(\hat{F}(x))\}$. If there exists an open convex set $\mathcal{C}_0\supset\mathcal{C}$ such that $\mathcal{F}:\mathcal{C}_0\rightarrow\mathcal{L}_n$ is  continuously differentiable and pseudomonotone on $\mathcal{C}_0$, then the condition (\ref{eq4}) can be simplified as
\begin{equation*}
\left<y,D{\hat{F}}(x)y\right>\geq\alpha(x){\Vert y\Vert}^2,\ \forall y\in S',\ \mbox{when}\ x\in \phi(\mathcal{C}).
\end{equation*}
\end{remark}

{ We now give a numerical example to illustrate the above criteria. To avoid the unnecessary confusion brought by the stages and the cardinality of the sample space $\Xi$, we assume that stage $N$ and the cardinality of $\Xi$ are both $1$.
\begin{example}\label{exa1}
	Let $$
	\begin{array}{rcl}&& F(x_1,x_2)=(x_1,-x_2)^T,\quad \mathcal{N}=\{(l,0)^T:l\in \mathbb{R}\},\\
		&&\mathcal{C}=\{(x_1,x_2)^T:x_1\geq2,x_2\geq2,x_2-x_1\geq0\}.
	\end{array}$$  Then SVI (\ref{eq1}) is exhibited as
	$$-F(x_1,x_2)\in N_{\mathcal{C}\cap \mathcal{N}}(x_1,x_2).$$
\end{example}}

Obviously, $DF(x_1,x_2))=[1,0;0,-1]$. The orthogonal complement of $\mathcal{N}$ is $\mathcal{M}=\{(0,t)^T:t\in \mathbb{R}\}$ and the corresponding projection matrix $P_{\mathcal{M}}=[0,0;0,1]$.

We now prove that $F(x_1,x_2)$ is pseudomonotone on $\mathcal{C}$. Construct $$\mathcal{C}_0=\{(x_1,x_2)^T:x_1>1,x_2>1,x_2-2x_1>0\}.$$ Then $\mathcal{C}_0\supset \mathcal{C}$ is an open convex set. For $(x_1,x_2)^T\in \mathcal{C}_0$, we have $F(x_1,x_2)\neq0$ and $${\rm Ker}(F(x_1,x_2))=\{(u_1,u_2)^T:u_{2}=\frac{x_1}{x_2}u_{1}\}.$$ Given $(u_1,u_2)^T\in {\rm Ker}(F(x_1,x_2))$, we have
\begin{equation*}
	(u_1,u_2)DF(x_1,x_2)(u_1,u_2)^T={u_1}^2-{u_2}^2={u_1}^2-{(\frac{x_1}{x_2})}^2{u_1}^2>0.
\end{equation*}
By Lemma \ref{lem4}, $F(x_1,x_2)$ is pseudomonotone on $\mathcal{C}$. But $F(x_1,x_2)$ is not monotone since $DF(x_1,x_2))$ has an negative eigenvalue.

Given $x\in \mathcal{C}$, we have $\left<u,DF(x)u\right>\geq{\Vert u\Vert}^2$ for $u\in \mathcal{N}$. Thus $\alpha(x)=1$, $\beta(x)=0$ and $\gamma(x)=1$ in Theorem \ref{thm5}. Then the monotonicity of $F+N_{\mathcal{C}}$ is globally elicited at level $s\geq e_0=1$ via Theorem $\ref{thm5}$.

\begin{theorem}\label{thm6}
Consider SVI (\ref{eq1}). Let $\mathcal{C}_0\supset\mathcal{C}$ be an open convex set in $\mathcal{L}_n$, and  $\mathcal{F}:\mathcal{C}_0\rightarrow\mathcal{L}_n$ be continuously differentiable and pseudomonotone on $\mathcal{C}_0$. Assume that $\hat{F}$ defined in (\ref{eq2.5}) satisfies the following conditions:
\begin{equation}\label{eq6}
\begin{array}{l}
(i)\ \text{$D\hat{F}(x)$ is symmetric}; \\
(ii)\ D\hat{F}(x)P_{\phi(\mathcal{M}_n)}=P_{\phi(\mathcal{M}_n)}D\hat{F}(x)
\end{array}
\end{equation}
for any $x\in\phi(\mathcal{C}_0)$, and
\begin{equation*}
(iii)\ \lambda_{i}(P_{\phi(\mathcal{M}_n)})>0\ \text{if $\lambda_{i}(D\hat{F}(x))<0$ for some $x\in\phi(\mathcal{C}_0)$}.
\end{equation*}
Let
\begin{equation*}
e_1=\inf\limits_{x\in\phi(\mathcal{C}_0)}\lambda_{\bar{n}}^{\downarrow}(D\hat{F}(x)).
\end{equation*} If $e_1>-\infty$, then the monotonicity of $\mathcal{F}+N_{\mathcal{C}}$ is globally elicited at level $s\geq-e_1$.
\end{theorem}
\begin{proof}
According to Lemma \ref{lem5}, the eigenvalues of $D\hat{F}(x)+sP_{\phi(\mathcal{M}_n)}$ can be denoted as $\lambda_i(D\hat{F}(x))+s\lambda_i(P_{\phi(\mathcal{M}_n)})$ for $x\in\phi(\mathcal{C}_0)$ and $1\leq i\leq\bar{n}$. Inasmuch as the eigenvalues of projection matrix is 0 or 1, $\lambda_i(D\hat{F}(x))+s\lambda_i(P_{\phi(\mathcal{M}_n)})\geq0$ when $\lambda_{i}(D\hat{F}(x))\geq0$. If $\lambda_{i}(D\hat{F}(x))<0$ for some $x\in\phi(\mathcal{C}_0)$, it is not difficult to see that $\lambda_{i}(D\hat{F}(x))+s\lambda_{i}(P_{\phi(\mathcal{M}_n)})\geq0$ when $s\geq-e_1$. By \cite[Proposition 2.3.2]{Facchinei2003}, $\hat{F}+sP_{\phi(\mathcal{M}_n)}$ is monotone on $\phi(\mathcal{C}_0)$. Thus $\hat{F}+N_{\phi(\mathcal{C})}+sP_{\phi(\mathcal{M}_n)}$ is maximal monotone from \cite[Theorem 3]{Rockafellar1970}, which indicates the maximal monotonicity of $\mathcal{F}+N_{\mathcal{C}}+sP_{\mathcal{M}_n}$.
\end{proof}

\begin{example} Assume that stage $N$ and the cardinality of $\Xi$ are both 1. Let $F(x_1,x_2)=(0,-x_2)^T$, $\mathcal{C}=\{(x_1,x_2)^T:x_1\geq1,x_2\geq1\}$, and $\mathcal{N}=\{(l,0)^T:l\in \mathbb{R}\}$. Then SVI (\ref{eq1}) is exhibited as
		$$-F(x_1,x_2)\in N_{\mathcal{C}\cap \mathcal{N}}(x_1,x_2).$$
\end{example}
		
		Obviously, $DF(x_1,x_2))=[0,0;0,-1]$, the orthogonal complement of $\mathcal{N}$ is $\mathcal{M}=\{(0,t)^T:t\in R\}$ and the corresponding projection matrix is $P_{\mathcal{M}}=[0,0;0,1]$.
		
		Construct $\mathcal{C}_0=\{(x_1,x_2)^T:x_1>0,x_2>0\}$. Then $\mathcal{C}_0\supset \mathcal{C}$ is an open convex set.  It holds that $F(x_1,x_2)$ is pseudomonotone on $\mathcal{C}_0$. In fact, for any given $(x_1,x_2)^T,(y_1,y_2)^T\in \mathcal{C}_0$ with $\left<-F(x_1,x_2),(y_1-x_1,y_2-x_2)^T\right>\geq0$, we have $y_2-x_2\leq0$ due to
		$$\left<-F(x_1,x_2),(y_1-x_1,y_2-x_2)^T\right>=-x_2(y_2-x_2), \quad x_2>0.$$
		This together with $y_2\ge 0$ implies that
		$$\left<-F(y_1,y_2),(y_1-x_1,y_2-x_2)^T\right>=-y_2(y_2-x_2)\geq0.$$  Hence, $F(x_1,x_2)$ is pseudomonotone on $\mathcal{C}_0$. However, $F(x_1,x_2)$ is not monotone on $\mathcal{C}_0$ because of the negative eigenvalue of $DF(x_1,x_2)$. Nevertheless, for $(x_1,x_2)^T\in \mathcal{C}$, $DF(x_1,x_2)$ is symmetric, and $DF(x_1,x_2)P_{\mathcal{M}}=P_{\mathcal{M}}DF(x_1,x_2)$. Besides, $\lambda_i(P_{\mathcal{M}})=1>0$ while $\lambda_i(DF(x))<0$. Thus,  by Theorem \ref{thm6},  $F+N_{\mathcal{C}}+sP_{\mathcal{M}}$ is globally maximal monotone at level $s\geq e_1=1$.

\begin{remark} \label{rem6-3}
Note that the condition $(ii)$ in (\ref{eq6}) may hold true even when $D\hat{F}(x)$ and $P_{\phi(\mathcal{M}_n)}$ are not diagonal matrices. For example, assume that stage $N$ and the cardinality of $\Xi$ are both 1, $C=\mathbb{R}_{+}^4$, $\hat{F}=Mx$ with
\begin{equation*}
M=
\begin{bmatrix}
10&0&0&1\\
0&4&1&0\\
0&1&4&0\\
1&0&0&5
\end{bmatrix},
\end{equation*}
and subspace $M_4=P_{M_4}x$ with
\begin{equation*}
P_{M_4}=
\begin{bmatrix}
1&0&0&0\\
0&0.5&0.5&0\\
0&0.5&0.5&0\\
0&0&0&1
\end{bmatrix}.
\end{equation*}
$P_{M_4}$ is a projection matrix due to $P_{M_4}^2=P_{M_4}$ and $P_{M_4}^T=P_{M_4}$. Since $M$ is positive semidefinite, $\hat{F}$ is monotone on $C$ and thus pseudomonotone on $C$. Via some simple calculations, we can see that $D\hat{F}(x)P_{M_4}=P_{M_4}D\hat{F}(x)$ for $x\in C$, as is consistent with condition $(ii)$ in (\ref{eq6}).
\end{remark}

In case that it is difficult to calculate the minimum eigenvalue $\lambda_{\bar{n}}^{\downarrow}(D\hat{F}(x))$ over $\phi({\cal C}_0)$, the following corollary says that we may use the spectral radius to replace the minimum eigenvalue.
\begin{corollary}\label{cor6-1}
Consider SVI (\ref{eq1}). Let $\mathcal{C}_0\supset\mathcal{C}$ be an open convex set in $\mathcal{L}_n$, and  $\mathcal{F}:\mathcal{C}_0\rightarrow\mathcal{L}_n$ be continuously differentiable and pseudomonotone on $\mathcal{C}_0$. Assume that $\hat{F}$ defined in (\ref{eq2.5}) satisfies the following conditions for any $x\in\phi(\mathcal{C}_0)$:
\begin{equation}\label{eq6-1-1}
\begin{array}{l}
(i)\ \text{$D\hat{F}(x)$ is symmetric}; \\
(ii)\ {\rm tr}((D\hat{F}(x)P_{\phi(\mathcal{M}_n)})^{2})={\rm tr}((D\hat{F}(x))^{2}P_{\phi(\mathcal{M}_n)});\\
(iii)\ \text{$D\hat{F}(x)P_{\phi(\mathcal{M}_n)}$ is not positive semidefinite.}
\end{array}
\end{equation}

Let $\hat{e}_1=\sup\limits_{x\in\phi(\mathcal{C}_0)}\rho(D\hat{F}(x))$.
If $\hat{e}_1<+\infty$, then the monotonicity of $\mathcal{F}+N_{\mathcal{C}}$ is globally elicited at level $s\geq \hat{e}_1$.
\end{corollary}
\begin{proof}
On account of the fact that $\rho(D\hat{F}(x))\geq|\lambda_{\bar{n}}^{\downarrow}(D\hat{F}(x))|$, it suffices to prove that the conditions (ii) and (iii) imply the conditions in Theorem \ref{thm6}.

Firstly, $D\hat{F}(x)$ commutes with $P_{\phi(\mathcal{M}_n)}$ if and only if $tr((D\hat{F}(x)P_{\phi(\mathcal{M}_n)})^{2})={\rm tr}((D\hat{F}(x))^{2}P_{\phi(\mathcal{M}_n)})$ \cite{Roger2013}. This is condition (ii).

Secondly, when $D\hat{F}(x)P_{\phi(\mathcal{M}_n)}$ is not positive semidefinite, condition $(iii)$ in Theorem \ref{thm6} holds true. In fact, suppose that there exist orthogonal matrix $U$ such that $U^{T}D\hat{F}(x)U=\Lambda_{A}$ and $U^{T}P_{\phi(\mathcal{M}_n)}U=\Lambda_{B}$, where $\Lambda_{A}$ and $\Lambda_{B}$ are diagonal matrices and the existence of $U$ is assured by Lemma \ref{lem5}. Then we have
\begin{equation*}
U^{T}D\hat{F}(x)UU^{T}P_{\phi(\mathcal{M}_n)}U=U^{T}D\hat{F}(x)P_{\phi(\mathcal{M}_n)}U=\Lambda_{A}\Lambda_{B}.
\end{equation*}
Thus $D\hat{F}(x)P_{\phi(\mathcal{M}_n)}$ and $\Lambda_{A}\Lambda_{B}$ have same eigenvalues. Since $D\hat{F}(x)P_{\phi(\mathcal{M}_n)}$ is not positive semidefinite, there exists at least one negative eigenvalue for $D\hat{F}(x)P_{\phi(\mathcal{M}_n)}$, and so is $\Lambda_{A}\Lambda_{B}$, as indicates condition $(iii)$ in Theorem \ref{thm6}.
\end{proof}

Now we give some remarks to explain that the conditions in Theorem \ref{thm6} may hold.

\begin{remark}\label{rem6-1}
Since $D\hat{F}(x)$ is block diagonal for $x\in\phi(\mathcal{C})$, i.e.,
\begin{equation*}D\hat{F}(x)=
\begin{bmatrix}
DF(\frac{1}{\sqrt{p(\xi^1)}}x_{p_1},\xi^1) & & &0\\
& DF(\frac{1}{\sqrt{p(\xi^2)}}x_{p_2},\xi^2) & &\\
& &\ddots& \\
0& & & DF(\frac{1}{\sqrt{p(\xi^J)}}x_{p_J},\xi^J)
\end{bmatrix},
\end{equation*}
there also exist some other criteria for condition $D\hat{F}(x)P_{\phi(\mathcal{M}_n)}=P_{\phi(\mathcal{M}_n)}D\hat{F}(x)$.

For instance, provided that $P_{\phi(\mathcal{M}_n)}$ is block diagonal, i.e., \begin{equation*}P_{\phi(\mathcal{M}_n)}=
\begin{bmatrix}
P_{1} & & &0\\
& P_{2} & &\\
& &\ddots& \\
0& & & P_{J}
\end{bmatrix},
\end{equation*}
and $$DF(\frac{1}{\sqrt{p(\xi^i)}}x_{p_i},\xi^i)P_{i}=P_{i}DF(\frac{1}{\sqrt{p(\xi^i)}}x_{p_i},\xi^i), \quad 1\leq i\leq J.$$ Then, it is easy to see that $D\hat{F}(x)P_{\phi(\mathcal{M}_n)}=P_{\phi(\mathcal{M}_n)}D\hat{F}(x)$.
\end{remark}

\begin{remark}\label{rem6-2}
Since projection matrix $P_{\phi(\mathcal{M}_n)}$ is usually sparse, the computation cost of the principal minors (or the eigenvalues) of matrix $D\hat{F}(x)P_{\phi(\mathcal{M}_n)}$ is relatively low, especially when $\hat{F}$ is linear. In this case, it is easy to test  whether $D\hat{F}(x)P_{\phi(\mathcal{M}_n)}$ is positive semidefinite since the matrix is constant.
\end{remark}

\begin{theorem}\label{thm7}
Consider SVI (\ref{eq1}). Let $\mathcal{C}_0\supset\mathcal{C}$ be an open convex set in $\mathcal{L}_n$, and  $\mathcal{F}:\mathcal{C}_0\rightarrow\mathcal{L}_n$ be continuously differentiable and pseudomonotone on $\mathcal{C}_0$. Let $D\hat{F}(x)$ be symmetric with $\hat{F}$ defined in (\ref{eq2.5}) for $x\in\phi(\mathcal{C}_0)$. If $e_{2}>0$ satisfies the condition that the multiplicity of the minimum eigenvalue of $D\hat{F}(x)+e_{2}P_{\phi(\mathcal{M}_n)}$ is strictly larger than one for all $x\in\phi(\mathcal{C}_0)$, then the monotonicity of $\mathcal{F}+N_{\mathcal{C}}$ is globally elicited at level $s\geq e_{2}$.
\end{theorem}
\begin{proof}
Let $A=D\hat{F}(x)+e_{2}P_{\phi(\mathcal{M}_n)}$, $B=e_{2}P_{\phi(\mathcal{M}_n)}$ and $j=\bar{n}-1$, $i=\bar{n}$. It follows from Lemma \ref{lem6} that $$\lambda_{\bar{n}-1}^{\downarrow}(D\hat{F}(x)+e_{2}P_{\phi(\mathcal{M}_n)})\geq\lambda_{\bar{n}}^{\downarrow}(e_{2}P_{\phi(\mathcal{M}_n)})+\lambda_{\bar{n}-1}^{\downarrow}(D\hat{F}(x)).$$ By Lemma \ref{lem4}, $\lambda_{\bar{n}-1}^{\downarrow}(D\hat{F}(x))\geq0$. Based on the hypothesis on $e_{2}$, we have $$\lambda_{\bar{n}}^{\downarrow}(D\hat{F}(x)+e_{2}P_{\phi(\mathcal{M}_n)})=\lambda_{\bar{n}-1}^{\downarrow}(D\hat{F}(x)+e_{2}P_{\phi(\mathcal{M}_n)})\geq0.$$ By \cite[Corollary 4.3.15]{Roger2013}, when $s\ge e_2$, it holds $$\lambda_{\bar{n}}^{\downarrow}(D\hat{F}(x)+sP_{\phi(\mathcal{M}_n)})\geq\lambda_{\bar{n}}^{\downarrow}(D\hat{F}(x)+e_{2}P_{\phi(\mathcal{M}_n)}),$$ which indicates the monotonicity of $\hat{F}+sP_{\phi(\mathcal{M}_n)}$ on $\phi(\mathcal{C}_0)$. Via the same procedure in the proof of Theorem \ref{thm6}, $\mathcal{F}+N_{\mathcal{C}}+sP_{\mathcal{M}_n}$ is maximal monotone.
\end{proof}
{
	\begin{example}
		Assume that stage $N$ and the cardinality of $\Xi$ are both 1. Let $F(x_1,x_2,x_3)=(-x_1,0,0)^T$, $\mathcal{N}=\{(0,l,t)^T:l,t\in \mathbb{R}\}$, and $\mathcal{C}=\{(x_1,x_2,x_3)^T:x_1\geq1,x_2\geq1,x_3\geq1\}$. Then SVI (\ref{eq1}) is exhibited as $-F(x_1,x_2,x_3)\in N_{\mathcal{C}\cap \mathcal{N}}(x_1,x_2,x_3)$.
\end{example}}
		
		Obviously, $DF(x_1,x_2,x_3))=[-1,0,0;0,0,0;0,0,0]$. In addition, the orthogonal complement of $\mathcal{N}$ is $\mathcal{M}=\{(v,0,0)^T:v\in \mathbb{R}\}$ and the corresponding projection matrix $P_{\mathcal{M}}=[1,0,0;0,0,0;0,0,0]$.
		
		Construct $\mathcal{C}_0=\{(x_1,x_2,x_3)^T:x_1>0,x_2>0,x_3>0\}$. Then $\mathcal{C}_0\supset \mathcal{C}$ is an open convex set. Obviously, $F(x_1,x_2,x_3)$ is pseudomonotone on $\mathcal{C}_0$. In fact, for $(x_1,x_2,x_3)^T,(y_1,y_2,y_3)^T\in \mathcal{C}_0$ with $$\left<-F(x_1,x_2,x_3),(y_1-x_1,y_2-x_2,y_3-x_3)^T\right>\geq0,$$ we have $y_1-x_1\leq0$ since $$\left<-F(x_1,x_2,x_3),(y_1-x_1,y_2-x_2,y_3-x_3)^T\right>=-x_1(y_1-x_1),\quad x_1>0.$$ This implies that $$\left<-F(y_1,y_2,y_3),(y_1-x_1,y_2-x_2,y_3-x_3)^T\right>=-y_1(y_1-x_1)\geq0.$$ Thus $F(x_1,x_2,x_3)$ is pseudomonotone on $\mathcal{C}_0$.
		
		On the other hand, $F(x_1,x_2,x_3)$ is not monotone on $\mathcal{C}_0$ as a result of the negative eigenvalue of $DF(x_1,x_2,x_3)$. But $F+N_{\mathcal{C}}+sP_{\mathcal{M}}$ is globally maximal monotone at level $s\geq1$ from Theorem $\ref{thm7}$, inasmuch as $DF(x_1,x_2,x_3)$ is symmetric and the multiplicity of the minimum eigenvalue of $DF(x_1,x_2,x_3)$ is $2$.

\begin{corollary}\label{cor7-1}
Consider SVI (\ref{eq1}). Let $\mathcal{C}_0\supset\mathcal{C}$ be an open convex set in $\mathcal{L}_n$, and  $\mathcal{F}:\mathcal{C}_0\rightarrow\mathcal{L}_n$ be continuously differentiable and pseudomonotone on $\mathcal{C}_0$. Let $D\hat{F}(x)$ be symmetric with $\hat{F}$ defined in (\ref{eq2.5}) for $x\in\phi(\mathcal{C}_0)$. If $\hat{e}_2>0$ satisfies the condition that $D\hat{F}(x)+\hat{e}_2P_{\phi(\mathcal{M}_n)}$ is block diagonal with $k(x)$ pairs of same blocks for all $x\in\phi(\mathcal{C}_0)$, where $k(x)\geq1$ is an integer, i.e.,
\begin{equation*}D\hat{F}(x)+\hat{e}_2P_{\phi(\mathcal{M}_n)}=
\begin{bmatrix}
A_{1} & & &0\\
& A_{2} & &\\
& &\ddots& \\
0& & & A_{2k(x)}
\end{bmatrix},
\end{equation*}
and there is partition $\{i_{1},i_{2},\ldots,i_{k(x)}\}$ and $\{j_{1},j_{2},\ldots,j_{k(x)}\}$ of $\{1,\ldots,2k(x)\}$ such that $A_{i_{l}}=A_{j_{l}}$, $1\leq l\leq k(x)$, then the monotonicity of $\mathcal{F}+N_{\mathcal{C}}$ is globally elicited at level $s\geq \hat{e}_2$.
\end{corollary}
\begin{proof}
It suffices to prove that the multiplicity of every eigenvalue of $D\hat{F}(x)+\hat{e}_2P_{\phi(\mathcal{M}_n)}$ is strictly larger than one. Actually, the eigenvalues of $D\hat{F}(x)+\hat{e}_2P_{\phi(\mathcal{M}_n)}$ are the aggregation of the eigenvalues of all blocks. Since two same blocks have same eigenvalues, the multiplicity of every eigenvalue of $D\hat{F}(x)+\hat{e}_2P_{\phi(\mathcal{M}_n)}$ is strictly larger than one.
\end{proof}

\begin{theorem}\label{thm8}
Consider SVI (\ref{eq1}). Let $\mathcal{C}_0\supset\mathcal{C}$ be an open convex set in $\mathcal{L}_n$, and  $\mathcal{F}:\mathcal{C}_0\rightarrow\mathcal{L}_n$ be continuously differentiable and pseudomonotone on $\mathcal{C}_0$. Let $D\hat{F}(x)$ be symmetric with $\hat{F}$ defined in (\ref{eq2.5}) for $x\in\phi(\mathcal{C}_0)$. If $e_{3}>0$ satisfies the following conditions:
\begin{equation}\label{eq8-1}
\begin{array}{l}
(i)  (D\hat{F}(x))_{ii}+e_{3}(P_{\phi(\mathcal{M}_n)})_{ii}>0;\\
(ii)\ \text{$D\hat{F}(x)+e_{3}P_{\phi(\mathcal{M}_n)}$ is strictly diagonally dominant}, i.e.,
\end{array}
\end{equation}
\begin{equation*}
(D\hat{F}(x))_{ii}+e_{3}(P_{\phi(\mathcal{M}_n)})_{ii}>\sum\limits_{j\neq i}|(D\hat{F}(x))_{ij}+e_{3}(P_{\phi(\mathcal{M}_n)})_{ij}|
\end{equation*}
for $i=1,2,\ldots,\bar{n}$ and $x\in\phi(\mathcal{C}_0)$, then the monotonicity of $\mathcal{F}+N_{\mathcal{C}}$ is globally elicited at level $s\geq e_{3}$.
\end{theorem}
\begin{proof}
By Lemma \ref{std2}, the matrix $D\hat{F}(x)+e_{3}P_{\phi(\mathcal{M}_n)}$ is positive semidefinite. Hence, by \cite[Corollary 4.3.15]{Roger2013}, $D\hat{F}(x)+sP_{\phi(\mathcal{M}_n)}$ is positive semidefinite when $s\geq e_{3}$. Thus the monotonicity of $\mathcal{F}+N_{\mathcal{C}}$ is globally elicited at level $s\geq e_{3}$.
\end{proof}

Note that the projection matrices have nonnegative diagonal elements and are usually diagonally dominant. Define the index set
\begin{equation*}
I=\{1\leq i\leq\bar{n}:P_{\phi(\mathcal{M}_n)})_{ii}>0\}.
\end{equation*}
Let $P_{\phi(\mathcal{M}_n)}^{I}$ be the submatrix of $P_{\phi(\mathcal{M}_n)}$ with the $i$th row and the $i$th column  removed for $i\notin I$. We have the following result.
\begin{corollary}
Consider SVI (\ref{eq1}). Let $\mathcal{C}_0\supset\mathcal{C}$ be an open convex set in $\mathcal{L}_n$, and  $\mathcal{F}:\mathcal{C}_0\rightarrow\mathcal{L}_n$ be continuously differentiable and pseudomonotone on $\mathcal{C}_0$. Let $D\hat{F}(x)$ be symmetric with $\hat{F}$ defined in (\ref{eq2.5}) for $x\in\phi(\mathcal{C}_0)$. Assume that
\begin{equation*}
\begin{array}{rcl}
(i) && \text{$P_{\phi(\mathcal{M}_n)}^{I}$ is strictly diagonally dominant};\\
(ii)&& \text{the $i$th row  of $D\hat{F}(x)$ and $P_{\phi(\mathcal{M}_n)}$ are zero for all $i\notin I$ and $x\in\phi(\mathcal{C}_0)$.}
\end{array}
\end{equation*}
Define
\begin{equation*}
\hat{e}_3=\max\left\{\sup\limits_{x\in\phi(\mathcal{C}_0),i\in I}\frac{\sum\limits_{j\neq i}|(D\hat{F}(x))_{ij}|-(D\hat{F}(x))_{ii}}{(P_{\phi(\mathcal{M}_n)})_{ii}-\sum\limits_{j\neq i}|(P_{\phi(\mathcal{M}_n)})_{ij}|}, \sup\limits_{x\in\phi(\mathcal{C}_0),i\in I}\{-(D\hat{F}(x))_{ii}\}\right\}.
\end{equation*}
If $\hat{e}_{3}<+\infty$, then the monotonicity of $\mathcal{F}+N_{\mathcal{C}}$ is globally elicited at level $s>\hat{e}_3$.
\end{corollary}
\begin{proof}
If $s>\hat{e}_3$, then $D\hat{F}^{I}(x)+sP_{\phi(\mathcal{M}_n)}^{I}$ is positive semidefinite for $x\in\phi(\mathcal{C}_0)$ from Lemma \ref{std2}. Since the $i$th row  and the $i$th column  of $D\hat{F}(x)$ and $P_{\phi(\mathcal{M}_n)}$ are zero for all $i\notin I$ and $x\in\phi(\mathcal{C}_0)$, the matrix $D\hat{F}(x)+sP_{\phi(\mathcal{M}_n)}$ is positive semidefinite. Hence, the desired result holds.
\end{proof}

\section{Numerical experiments} \label{nume}
In this section, we demonstrate the effectiveness of the elicited PHA in solving a two-stage pseudomonotone stochastic linear complementarity problem (SLCP). The two-stage SLCP is given as a special case of
\begin{equation*}
-\mathcal{F}(x(\cdot))-w(\cdot)\in N_{\mathcal{C}}(x(\cdot)),
\end{equation*}
where $F(x(\xi),\xi)=M(\xi)x(\xi)+q(\xi)$ with $M(\xi)\in \mathbb{R}^{n\times n}$ and $q(\xi)\in \mathbb{R}^n$, $C(\xi)=\mathbb{R}^n_{+}$, and $x(\xi)=(x_{1}(\xi)^T,x_{2}(\xi)^T)^T$ with $x_{i}(\xi)\in \mathbb{R}^{n_{i}}$ being the $i$-th stage decision vector for $i=1,2$. In this model, the nonanticipativity subspace is described as $\mathcal{N}_n=\{x(\cdot)=(x_1(\cdot),x_2(\cdot)):x_1(\xi)\ \text{is same for all}\ \xi\in\Xi\}$, and the corresponding complement is $\mathcal{M}_n=\{w(\cdot)=(w_1(\cdot),w_2(\cdot)):E_{\xi}(w_{1}(\xi))=0,w_{2}(\xi)=0\ \text{for all}\ \xi\in\Xi\}$. By denoting
\begin{equation*}
M(\xi)=\left[
\begin{array}{cc}
M_{11}(\xi)&M_{12}(\xi)\\
M_{21}(\xi)&M_{22}(\xi)
\end{array}\right],\ q(\xi)=\left[\begin{array}{c}
q_{1}(\xi)\\q_{2}(\xi)
\end{array}\right]
\end{equation*}
for $\xi\in\Xi$, where $M_{ij}(\xi)\in \mathbb{R}^{n_i\times n_j}$, $q_i(\xi)\in \mathbb{R}^{n_i}$ with $1\leq i,j\leq2$, the extensive form of the two-stage SLCP is formulated as
\begin{equation} \label{eq8}
\left\{\begin{array}{l}
0\leq x_{1}\perp M_{11}(\xi)x_{1}+M_{12}(\xi)x_{2}(\xi)+q_{1}(\xi)+w_{1}(\xi)\geq0,\\
0\leq x_{2}(\xi)\perp M_{21}(\xi)x_{1}+M_{22}(\xi)x_{2}(\xi)+q_{2}(\xi)\geq0\quad \forall \xi\in\Xi.
\end{array}\right.
\end{equation}

In the execution of Algorithm $\ref{alg1}$, we use the semismooth Newton method \cite{QS} to solve (\ref{eqalg1}). It is worth mentioning that we take $\hat{x}^k(\xi)$ as the starting point  in the $k+1$ iteration, which is termed the ``warm start" feature of PHA in \cite{Rockafellar2019}. In addition, inspired by \cite{Zhang2019} and \cite{Eckstein1992}, we change the step size in the dual update to be
\begin{equation*}
w^{k+1}(\xi)=w^k(\xi)+\rho(r-s)({\hat{x}}^k(\xi)-x^{k+1}(\xi)),
\end{equation*}
where $\rho=1.618$, and $s$ is supposed to be $r/2$.

If $x(\xi)$ and $w(\xi)$ solve (\ref{eqalg1}), then $x(\xi)$ and $w(\xi)$ also satisfy
\begin{equation} \label{eq9}
\left\{\begin{array}{l}
0\leq x_{1}\perp E_{\xi}(M_{11}(\xi)x_{1})+E_{\xi}(M_{12}(\xi)x_{2}(\xi))+E_{\xi}(q_{1}(\xi))\geq0,\\
0\leq x_{2}(\xi)\perp M_{21}(\xi)x_{1}+M_{22}(\xi)x_{2}(\xi)+q_{2}(\xi)\geq0\quad \forall \xi\in\Xi,
\end{array}\right.
\end{equation}
as is obtained by taking the expectation in the first subproblem of (\ref{eq8}). Thus the stopping criterion can be designated as
\begin{equation*}
err=\max\{err_1,err_2\},
\end{equation*}
where
\begin{equation*}
err_1=\frac{\Vert x_1-\Pi_{\geq0}(x_1-(E_{\xi}(M_{11}(\xi)x_{1})+E_{\xi}(M_{12}(\xi)x_{2}(\xi))+E_{\xi}(q_{1}(\xi))))\Vert}{1+\Vert x_1\Vert},
\end{equation*}
\begin{equation*}
err_2=\max\limits_{\xi}\left\{\frac{\Vert x_2(\xi)-\Pi_{\geq0}(x_2(\xi)-( M_{21}(\xi)x_{1}+M_{22}(\xi)x_{2}(\xi)+q_{2}(\xi)))\Vert}{1+\Vert x_2(\xi)\Vert}\right\},
\end{equation*}
with $(\Pi_{\geq0}(a))_{j}=\max\{a_{j},0\}$. Set the tolerance to be $10^{-5}$ and the maximal iterations to be $2000$. In the next subsections, we apply Algorithm \ref{alg1} to solve a two-stage pseudomonotone SVI from real life and some randomly generated problems.

\subsection{Test on a two-stage orange market model}
Consider a two-stage orange market model in \cite{McCarl1977}, where the supply and demand curves are linear. Specifically, an orange firm mainly sells two kinds of products. One is the juice converted from oranges, and the other one is exactly the fresh oranges. Assume that the producer makes $1$ unit of juice from $2$ oranges. In the first stage, the firm has to decide the supply quantity ($Q_{S}$) with the supply price ($P_{S}$) determined as follows:
\begin{equation}\label{ora-1}
P_{S}=3+0.0005Q_{S}.
\end{equation}
In the second stage, the orange firm needs to determine the quantity of the juice ($Q_{J}$) and the quantity of fresh oranges ($Q_{F}$). Similarly, the price for orange juice ($P_{J}$) and price for fresh oranges ($P_{F}$) are related linearly to quantity $Q_{J}$ and $Q_{F}$, i.e.,
\begin{equation}\label{ora-2}
\begin{bmatrix}
P_{J}\\
P_{F}
\end{bmatrix}
=M
\begin{bmatrix}
Q_{J}\\
Q_{F}
\end{bmatrix}
+q,
\end{equation}
where $M\in \mathbb{R}^{2\times2}$ and $q\in \mathbb{R}^{2\times1}$. However, due to the uncertainties involving the climate, natural disaster, water resources and so on, linear relationship (\ref{ora-2}) may not be fixed. In our setting, three scenarios of uncertainties, $\xi^{1}$, $\xi^{2}$, $\xi^{3}$, with respective probabilities, 0.5, 0.3, 0.2, are considered. Then we denote the quantity variables $Q_{J}$ and $Q_{F}$ and price variables $P_{J}$ and $P_{F}$ for every scenario $\xi\in\Xi$ by $Q_{J}(\xi)$, $Q_{F}(\xi)$ and $P_{J}(\xi)$, $P_{F}(\xi)$ respectively, where $\Xi:=\{\xi^{1},\xi^{2},\xi^{3}\}$. Define $M$ and $q$, for $\xi\in\Xi$, as
\begin{equation}\label{ora-3}
\begin{aligned}
M(\xi^{1})=
\begin{bmatrix}
-0.005& -0.0002\\
-0.0002& -0.001
\end{bmatrix},
q(\xi^{2})=
\begin{bmatrix}
7.5\\
4
\end{bmatrix},\\
M(\xi^{2})=
\begin{bmatrix}
-0.004& -0.0001\\
-0.0001& -0.0005
\end{bmatrix},
q(\xi^{2})=
\begin{bmatrix}
7\\
3.5
\end{bmatrix},\\
M(\xi^{3})=
\begin{bmatrix}
-0.006& -0.0003\\
-0.0003& -0.0015
\end{bmatrix},
q(\xi^{3})=
\begin{bmatrix}
8\\
4.5
\end{bmatrix}.
\end{aligned}
\end{equation}
We build up the following two-stage optimization model:
\begin{equation}\label{ora-4}
\begin{aligned}
\min \quad& P_{S}Q_{S}+E_{\xi}[\phi(Q_{S},Q_{J}(\xi),Q_{F}(\xi))]\\
\mbox{s.t.}\quad& Q_{S}\geq0,
\end{aligned}
\end{equation}
where $\phi(Q_{S},Q_{J}(\xi),Q_{F}(\xi))$ is the optimal value of the second-stage optimization:
\begin{equation}\label{ora-5}
\begin{aligned}
\min\quad & -P_{J}(\xi)Q_{J}(\xi)-P_{F}(\xi)Q_{F}(\xi)\\
\mbox{s.t.}\quad& 2Q_{J}(\xi)+Q_{F}(\xi)\leq Q_{S}.
\end{aligned}
\end{equation}
Substituting (\ref{ora-1})-(\ref{ora-3}) into (\ref{ora-4})-(\ref{ora-5}), we get
\begin{equation}\label{ora-6}
\begin{aligned}
\min \quad& 0.0005Q_{S}^{2}+3Q_{S}+E_{\xi}[\phi(Q_{S},Q_{J}(\xi),Q_{F}(\xi))]\\
\mbox{s.t.}\quad& Q_{S}\geq0,
\end{aligned}
\end{equation}
where $\phi(Q_{S},Q_{J}(\xi),Q_{F}(\xi))$ is the optimal value of the second-stage optimization:
\begin{equation}\label{ora-7}
\begin{aligned}
\min\quad &
-
\begin{bmatrix}
Q_{J}(\xi)& Q_{F}(\xi)
\end{bmatrix}
M(\xi)
\begin{bmatrix}
Q_{J}(\xi)\\
Q_{F}(\xi)
\end{bmatrix}
-q(\xi)^{T}
\begin{bmatrix}
Q_{J}(\xi)\\
Q_{F}(\xi)
\end{bmatrix}\\
\mbox{s.t.}\quad& Q_{S}-
\begin{bmatrix}
2&1
\end{bmatrix}
\begin{bmatrix}
Q_{J}(\xi)\\
Q_{F}(\xi)
\end{bmatrix}
\geq0.
\end{aligned}
\end{equation}
The necessary  optimality condition of  (\ref{ora-6})-(\ref{ora-7}) is that there exist $\omega(\cdot)$ and dual vector $\eta(\cdot)$ such that, for $\xi\in\Xi$, the following condition holds \cite{Zhang2019}:
{\footnotesize \begin{equation}\label{ora-8}
\begin{aligned}
0\leq
\begin{bmatrix}
Q_{S}\\
Q_{J}(\xi)\\
Q_{F}(\xi)\\
\eta(\xi)
\end{bmatrix}
\perp
\begin{bmatrix}
0.001&0&0&-1\\
0&2M_{11}(\xi)&2M_{12}(\xi)&2\\
0&2M_{21}(\xi)&2M_{22}(\xi)&1\\
1&-2&-1&0
\end{bmatrix}
\begin{bmatrix}
Q_{S}\\
Q_{J}(\xi)\\
Q_{F}(\xi)\\
\eta(\xi)
\end{bmatrix}+
\begin{bmatrix}
3\\
q_{1}(\xi)\\
q_{2}(\xi)\\
0
\end{bmatrix}+
\begin{bmatrix}
\omega(\xi)\\
0\\
0\\
0
\end{bmatrix}
\geq0,
\end{aligned}
\end{equation}}
which is exactly a two-stage SLCP with $x_{1}(\xi)=Q_{S}$, $x_{2}(\xi)=[Q_{J}(\xi),Q_{F}(\xi),\eta(\xi)]^{T}$, $C(\xi)=\mathbb{R}_{+}^{4}$, and
\begin{equation*}
F(x(\xi),\xi)=
\begin{aligned}
\begin{bmatrix}
0.001&0&0&-1\\
0&2M_{11}(\xi)&2M_{12}(\xi)&2\\
0&2M_{21}(\xi)&2M_{22}(\xi)&1\\
1&-2&-1&0
\end{bmatrix}
x(\xi)+
\begin{bmatrix}
3\\
q_{1}(\xi)\\
q_{2}(\xi)\\
0
\end{bmatrix}.
\end{aligned}
\end{equation*}

Obviously, $\mathcal{F}(x(\cdot))$ is pseudomonotone, and its monotonicity can be globally elicited at any level $s\ge0$ via Theorem \ref{thm5}. We use Algorithm \ref{alg1} to solve (\ref{ora-8}). The algorithm ends up with the solution: $Q_{S}=393$, $Q_{J}(\xi^{1})=56$, $Q_{F}(\xi^{1})=281$, $Q_{J}(\xi^{2})=64$, $Q_{F}(\xi^{2})=265$, $Q_{J}(\xi^{3})=52$, $Q_{F}(\xi^{3})=288$. The corresponding price solution is $P_{S}=4.96$, $P_{J}(\xi^{1})=7.16$, $P_{F}(\xi^{1})=3.71$, $P_{J}(\xi^{2})=6.72$, $P_{F}(\xi^{2})=3.36$, $P_{J}(\xi^{3})=7.60$, $Q_{F}(\xi^{3})=4.05$.

\subsection{Test on randomly generated problems}
We implement Algorithm \ref{alg1} to solve the randomly generated numerical SLCPs. Assume that the space $\Xi$ has $J$ scenarios denoted as $\xi^{1},\xi^{2},\ldots,\xi^{J}$. Then the two-stage pseudomonotone SLCPs are generated randomly based on Corollary 6.6.2 in \cite{Cambini2009}.
\begin{itemize}
  \item Set $M(\xi^1)=ab^T+ba^T$, $q(\xi^1)=b_0a+a_0b+c$, and $c=\alpha a+\beta b$, where $a,b\in\mathbb{R}^n$ with $a\geq0,b\leq0$ are linearly independent, $a_0,b_0\in \mathbb{R}$ are randomly generated, and $\alpha,\beta$ is arbitrarily selected as long as $\alpha<-b_{0}$ and $\beta>-a_0$.
  \item Set $s=\lceil 3n/4\rceil$. Randomly generate number $a_{i}(\xi^{k})>0$ and vector $v_{i}(\xi^{k})\in \mathbb{R}^{n}$ for $i=1,2,\ldots,s$ and $k=2,\ldots,J$. Let $M(\xi^{k})=\sum\limits^{s}_{i=1}a_{i}(\xi^{k})v_{i}(\\\xi^{k})v_{i}(\xi^{k})^T$. Randomly generate $q(\xi^{k})\in \mathbb{R}^n$ for $k=2,\ldots,J$.
  \item Randomly generate the probabilities $p(\xi^{k})>0$ for $k=1,2,\ldots,J$.
\end{itemize}

We test three groups of two-stage pseudomonotone SLCPs listed as follows:
\begin{itemize}
  \item[{\bf G1:}] The dimensions of the problems ($dim$ for short) are set to be $[40,20]$. The number of the scenarios ($sn$ for short) in sample space $\Xi$ is increased from $50$ to $400$. 10 numerical examples are randomly generated for every setting of the problems. The numerical results including the average convergence iteration number (avg-iter for short) and average convergence iteration time (avg-time for short) are presented in Table \ref{table1} and Fig. \ref{fig1}.
  \item[{\bf G2:}] The number of the scenarios in sample space $\Xi$ is set to be $50$. The dimensions of the problems are increased from $[50,50]$ to $[400,400]$. 10 numerical examples are randomly generated for every setting of the problems. The numerical results including the average convergence iteration number and average convergence iteration time are presented in Table \ref{table2} and Fig. \ref{fig2}.
  \item[{\bf G3:}] Since the choice of $r$  is crucial in terms of the effectiveness of Algorithm \ref{alg1} \cite{Rockafellar2019}, we set the same $r=1$ and $r=\sqrt{n_1+n_2}$ as \cite{Rockafellar2019}, and randomly generate $10$ numerical examples respectively for each setting of the former two groups of problems.
\end{itemize}

\begin{remark}
Note that $F(x(\xi),\xi)$ is pseudomonotone if $\mathcal{F}(x(\cdot))$ is pseudomonotone, but the converse may be not true. Then, the set of pseudomonotone mappings $\mathcal{F}(x(\cdot))$ is contained in the set of mappings $\mathcal{F}(x(\cdot))$ with $F(x(\xi),\xi)$ being pseudomonotone for $\xi\in\Xi$. So, it is reasonable to design the above numerical experiments.
\end{remark}

\begin{table}[h]
\centering
\caption{Numerical results for the change of $sn$ ($dim$=[40,20])}
\label{table1}
\begin{tabular}{c c c c c}
\toprule
\multirow{2}{*}{$sn$} & \multicolumn{2}{c}{$r=1$}&\multicolumn{2}{c}{$r=\sqrt{60}$} \\
\cmidrule(r){2-3}\cmidrule(r){4-5}
& avg-iter& avg-time(s) &avg-iter &avg-time(s)\\
\hline
50  & 255.4 & 4.7  & 82.4 & 1.5\\
100 & 266.3 & 9.7  & 84.5 & 3.1\\
150 & 278.6 & 15.3 & 87.9 & 4.9\\
200 & 291.6 & 21.3 & 93.7 & 6.9\\
250 & 305.4 & 27.6 & 94.0   & 8.5\\
300 & 331.4 & 35.8 & 96.4 & 10.5\\
350 & 345.4 & 44.2 & 92.4 & 11.8\\
400 & 347.5 & 52.9 & 96.7 & 14.8\\
\bottomrule
\end{tabular}
\end{table}

\begin{figure} [h]
\begin{center}
\includegraphics[width=0.45\textwidth]{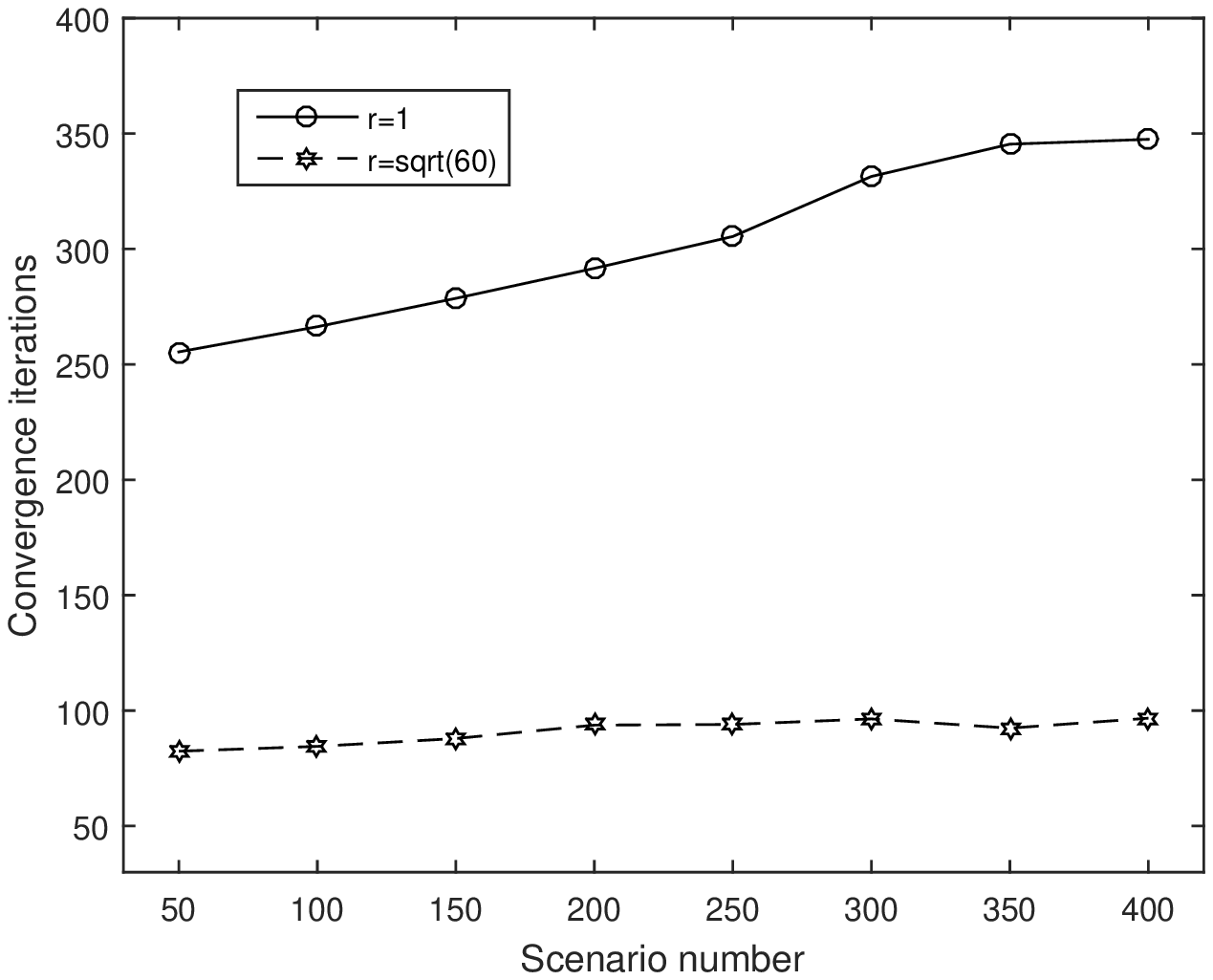}
\includegraphics[width=0.45\textwidth]{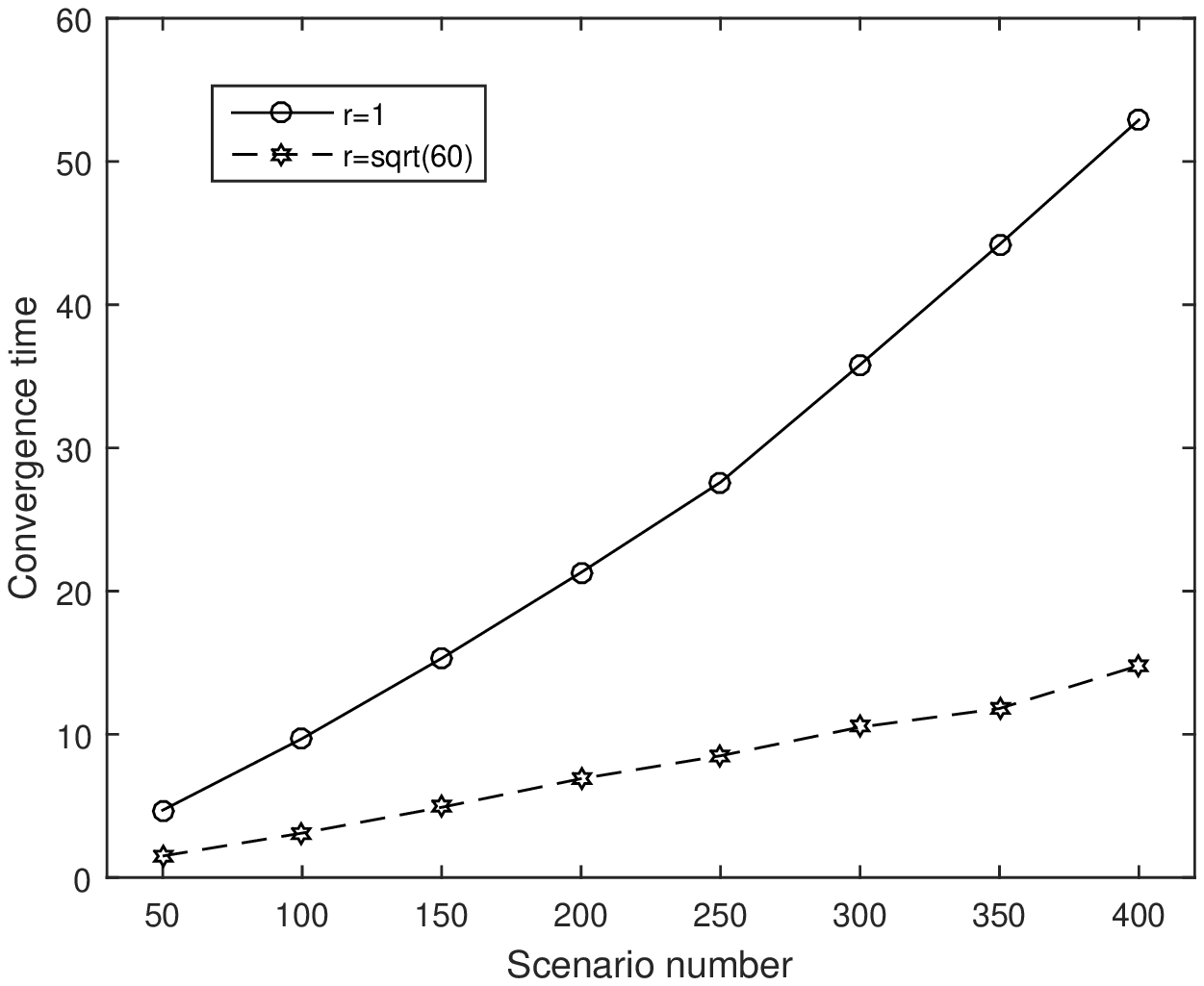}
\caption{Convergence results when $sn$ increases}
\label{fig1}
\end{center}
\end{figure}

\begin{table}[h]
\centering
\caption{Numerical results for the change of $dim$ ($sn$=50)}
\label{table2}
\begin{tabular}{c c c c c}
\toprule
$dim$ & \multicolumn{2}{c}{$r=1$}&\multicolumn{2}{c}{$r=\sqrt{n_1+n_2}$} \\
\cmidrule(r){2-3}\cmidrule(r){4-5}
& avg-iter& avg-time(s) &avg-iter &avg-time(s)\\
\hline
[50,50]   & 247.2  & 5.2   & 62.2 & 1.4\\{}
[100,100] & 473.4  & 13.7  & 37.4 & 1.2\\{}
[150,150] & 627.9  & 24.8  & 29.1 & 1.4\\{}
[200,200] & 840.0    & 62.9  & 27.7 & 2.6\\{}
[250,250] & 958.0   & 99.0   & 35.0   & 4.3\\{}
[300,300] & 1177.8 & 200.0   & 42.4 & 8.2\\{}
[350,350] & 1462.4 & 326.8 & 50.0   & 12.7\\{}
[400,400] & 1575.3 & 536   & 52.3 & 20.0\\
\bottomrule
\end{tabular}
\end{table}

\begin{figure} [h]
\centering
\includegraphics[width=0.45\textwidth]{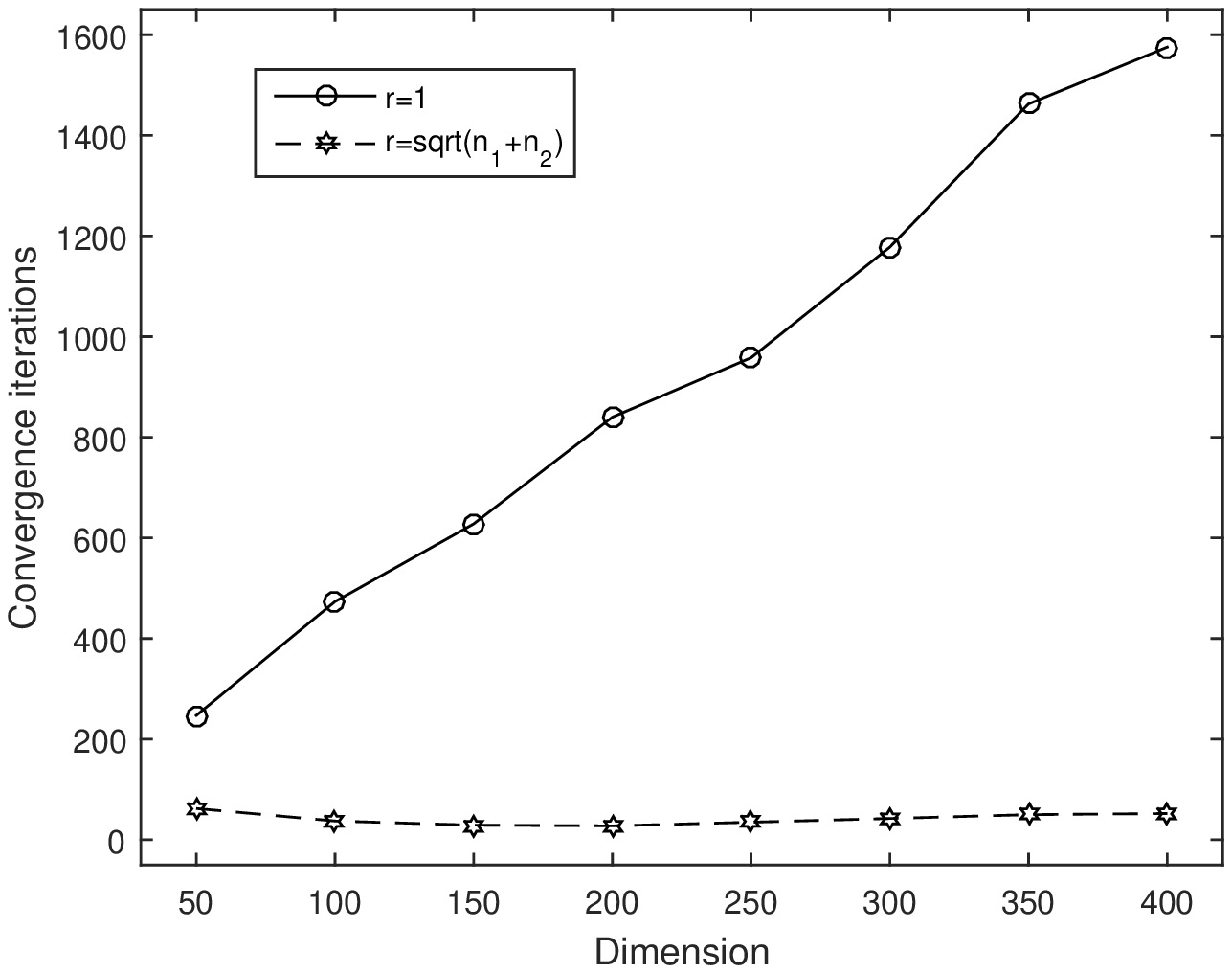}
\includegraphics[width=0.45\textwidth]{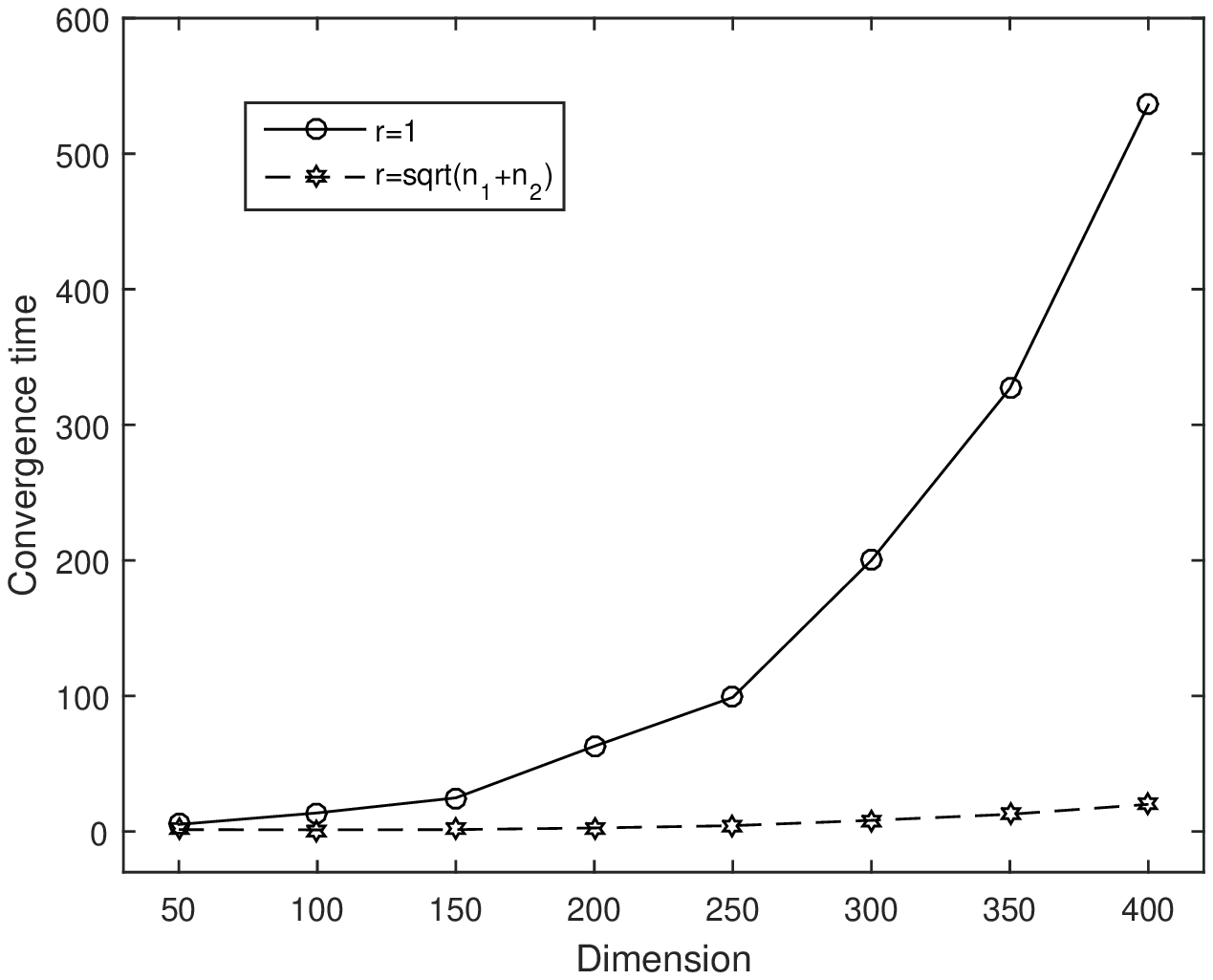}
\caption{Convergence results when $dim$ increases}
\label{fig2}
\end{figure}

We first fix the dimension of the decision vector to see the performance of Algorithm \ref{alg1}. From Table \ref{table1} and Fig.\ref{fig1},  we can see that when $r=1$, the iteration number and the computation time both grow roughly at a linear rate with the increasing of the scenario number. When $r=\sqrt{60}$, the computation time also grow roughly at a linear rate when the scenario number increases, while the growth rate of the iteration number is pretty slow. The performance of the elicited PHA with $r=\sqrt{60}$ is much better than that with $r=1$, which match with the numerical results given in \cite{Rockafellar2019}. Actually, the gap concerning the iteration number and the computation time between two choices of $r$ is widening with the increasing of the scenario number.

On the other hand, we fix the scenario number at $50$ and observe the influence of the dimension of the decision vector on the performance of Algorithm \ref{alg1}. From Table \ref{table2} and Fig.\ref{fig2}, the growth rates of the iteration number  and the computation time are steady with the increase of the scenario number when $r=1$. However, when $r=\sqrt{n_1+n_2}$, the iteration number remains stable and is around $40$. The computation time also grows very slowly when the scenario number increases. Similarly, the gap between two choices of $r$ is widening with the increasing of the dimension.

In summary, the above numerical experiments show that the elicited PHA is effective for the pseudomonotone two-stage SLCPs. Even for the relatively large cases, the elicited PHA can find the solution in a reasonable amount of time. When it comes to the choice of $r$, the performance of elicited PHA with $r=\sqrt{n_1+n_2}$ is much better than that with $r=1$, at least in our setting.

\begin{remark}
The results in our numerical experiments is basically consistent with the results in \cite[Table 2, 4]{Rockafellar2019}, but is slightly different from the results in \cite[Table 5, 6]{SZ2021}, where the numerical results show that increasing parameter $r$ leads to more iterations and convergence time. Nevertheless, it is worth noting that the settings of the experiments in our paper and \cite{SZ2021} are different, where $s=r/2$ in our paper while $s=r-1$ in \cite{SZ2021}. Actually, parameter $r$ with value $\sqrt{n_1+n_2}$ deserves to be explored further for its possible advantages.
\end{remark}

\section{Conclusions} \label{con}
We studied the multistage pseudomonotone SVI. First, we established some theoretical results on the existence, convexity, boundedness and compactness of the solution set based on constructing  the isomorphism between Hilbert space $\mathcal{L}_n$ and Euclidean space $R^{\bar{n}}$. Second, aiming at extension of the PHA from  monotone SVI problems to  nonmonotone ones,  we presented some sufficient conditions on the elicitability of the pseudomonotone SVIs, which opens the door for applying Rockafellar's elicited (nonmonotone) PHA to solve pseudomonotone SVIs. Numerical results on solving a pseudomonotone two-stage stochastic market optimization problem and experimental results on solving randomly generated two-stage SLCPs showed that the efficiency of the elicited PHA for solving pseudomonotone SVIs.

\bibliographystyle{siamplain}

\end{document}